\documentclass{article}
\usepackage[utf8]{inputenc}
\usepackage{srcltx}
\usepackage{amssymb,amsmath,amsfonts,amsthm}
\usepackage{cite}
\usepackage{verbatim}
\usepackage{hyperref}
\usepackage[english]{babel}
\usepackage{amsmath}
\usepackage{longtable,float}
\usepackage{tikz}
\usepackage{amssymb}
\usepackage{amsfonts}
\usepackage{textcomp}
\usepackage{enumerate} 

\newtheorem{theorem}{Theorem}[section]
\newtheorem{lemma}[theorem]{Lemma}
\newtheorem{prop}[theorem]{Proposition}
\newtheorem{remark}[theorem]{Remark}
\newcommand{\F}{\mathbb{F}}
\newcommand{\Z}{\mathbb{Z}}

\theoremstyle{definition}

\numberwithin{equation}{section}

\DeclareMathOperator{\SL}{SL}

\DeclareMathOperator{\sgn}{sgn}

\newcommand*{\No}{\textnumero}

\title{Minimal supplements in normalizers of maximal tori of $F_4(q)$}

\author{Alexey Galt\footnote{The first author is supported by Russian Scientific Fund (project №19-11-00039).}\, and Alexey Staroletov}

\date{\vspace{-5ex}}
\begin{document}
\newcommand{\Addresses}{{
  \bigskip
  \footnotesize

  A.~Galt, \textsc{Sobolev Institute of Mathematics, Novosibirsk, Russia;}\par\nopagebreak
  \textsc{Novosibirsk State University, Novosibirsk, Russia;}\par\nopagebreak
  \textit{E-mail address: } \texttt{galt84@gmail.com}

  \medskip

  A.~Staroletov, \textsc{Sobolev Institute of Mathematics, Novosibirsk, Russia;}\par\nopagebreak
  \textsc{Novosibirsk State University, Novosibirsk, Russia;}\par\nopagebreak
   \textit{E-mail address: } \texttt{staroletov@math.nsc.ru}
}}


\maketitle
\begin{abstract}
Let $G$ be a finite group of Lie type $F_4$ with the Weyl group $W$. 
For every maximal torus $T$ of $G$, we find the minimal order of a supplement to $T$ in its algebraic normalizer $N(G,T)$. In particular, we obtain all maximal tori having complements in $N(G,T)$. Assume that $T$ corresponds to an element $w$ of $W$. We find the minimal order of lifts of $w$ to $N(G,T)$.
\end{abstract}

{\bf keywords:} Finite group of Lie type $F_4$, Maximal torus, Algebraic normalizer, Weyl group, Minimal supplement

\section{Introduction}

Let $\overline{G}$ be a simple connected linear algebraic group over an algebraic closure $\overline{\F}_p$ of a finite field of positive characteristic $p$. Consider a Steinberg endomorphism $\sigma$ and a maximal $\sigma$-invariant torus $\overline{T}$  of $\overline{G}$. It is well known that all maximal tori are conjugate in $\overline{G}$ and a quotient $N_{\overline{G}}(\overline{T})/\overline{T}$ is isomorphic to the Weyl group $W$ of $\overline{G}$. The natural question is determine when $N_{\overline{G}}(\overline{T})$ splits over $\overline{T}$.
This question was stated by Tits in~\cite{Tits}. The answer was given independently in~\cite{AdamsHe} and as a corollary of papers~\cite{Galt1,Galt2,Galt3,Galt4}. The same problem for Lie groups was solved in \cite{LieGroups}.

A similar question can be formulated for finite groups. Consider a finite group of Lie type $G$, that is $O^{p'}(\overline{G}_{\sigma})\leqslant G\leqslant\overline{G}_{\sigma}$. Suppose that 
$T=\overline{T}\cap G$ is a maximal torus of $G$ and $N(G,T)=N_{\overline{G}}(\overline{T})\cap G$ is its algebraic normalizer. Then the question is to describe groups $G$ and their maximal tori $T$ such that $N(G,T)$ splits over $T$. It has been completely solved for the groups of types $A_n$, $B_n$, $C_n$, $D_n$, $E_6$, $E_7$, and $E_8$ in~\cite{Galt2,Galt3,Galt4,GS,GS2}. 
Observe that for some Lie types the situation when a maximal torus has no complements arises quite often. In these cases, one can find supplements of minimal order existing for all maximal tori. 

Adams and He~\cite{AdamsHe} considered a related problem. Namely, it is natural to ask about the orders of lifts of $w\in W$ to $N_{\overline{G}}(\overline{T})$. They noticed that if $d$ is the order of $w$ then the minimal order of a lift of $w$ is either $d$ or $2d$. Clearly, if $\overline{T}$ has a complement in  $N_{\overline{G}}(\overline{T})$ then the minimal order is equal to $d$. In the case of Lie type $F_4$,
they proved that the normalizer does not split and found the minimal orders of lifts of elements belonging to so-called regular or elliptic conjugacy classes of $W$. In particular, they noticed that there exists an elliptic element of order four such that any its lift has order eight.

In this paper we consider finite groups $G$ of Lie type $F_4$ over a finite field of order $q$.
Denote the Weyl group of type $F_4$ by $W$ and a fundamental system of roots of the corresponding root system by $\Delta=\{r_1,r_2,r_3,r_4\}$. We write $w_i$ for the element of $W$ corresponding to the reflection in the hyperplane orthogonal to the $i$-th positive root $r_i$. We assume that $r_8=r_1+r_2+r_3$, $r_{16}=r_2+2r_3+2r_4$, and $r_{21}=r_1+2r_2+3r_3+2r_4$. Recall that there is a bijection between conjugacy classes of maximal tori in $G$ and $\sigma$-conjugacy classes of~$W$. We find minimal supplements to maximal tori in their algebraic normalizers.

\begin{theorem}\label{th:F4}
Let $G=F_4(q)$ with the Weyl group $W$ and $w_0$ be the central involution of $W$.
Suppose that a maximal torus $T$ of $G$ corresponds to an element $w$ of $W$
and $M$ is a supplement to $T$ in $N(G,T)$ of minimal order. 
Then $|M\cap T|\leq 8$ and the following statements hold.
\begin{itemize}
  \item[{\em (i)}] $|M\cap T|=1$ if and only if either $q$ is even or the order of $w$ does not divide $4$. In other words, $T$ has a complement in $N(G,T)$ only in these cases.
  \item[{\em (ii)}] $|M\cap T|=2$ if and only if $q$ is odd and either $w$ or $ww_0$ is conjugate to one of the following elements: $w_3$, $w_{16}w_3$, $w_3w_2$, $w_2w_1w_{16}$, or $w_{16}w_3w_2$.
  \item[{\em (iii)}] $|M\cap T|=4$ if and only if $q$ is odd and $w$ is conjugate to one of the following elements: $1$, $w_0$, $w_6w_3$, $w_8w_{16}w_3w_2$, $w_2$ with $q\equiv1\negthickspace\pmod4$, or $w_0w_2$ with $q\equiv3\negthickspace\pmod4$.
  \item[{\em (iv)}] $|M\cap T|=8$ if and only if either $q\equiv3\negthickspace\pmod4$ and $w$ is conjugate to $w_2$ or $q\equiv1\negthickspace\pmod4$ and $w$ is conjugate to $w_2w_0$.
  \end{itemize}
\end{theorem}

We find the minimal orders of lifts of Weyl group elements to corresponding algebraic normalizers.

\begin{theorem}\label{th2:F4}
Let $G=F_4(q)$ with the Weyl group $W$.
Consider a maximal torus $T$ of $G$ corresponding to an element $w$ of $W$. Then $w$ has a lift to $N(G,T)$ of order $|w|$ except the following cases:
\begin{itemize}
  \item[{\em (i)}] $q$ is odd and $w$ is conjugate to $w_{16}w_3w_2$ or $w_{21}w_8w_3w_2$.
  \item[{\em (ii)}] $q\equiv3\negthickspace\pmod4$ and $w$ is conjugate to $w_3w_2$ or $w_2w_1w_{16}$.
  \end{itemize}
\end{theorem}
We illustrate the results of these theorems in Tables~\ref{t:main} and \ref{lifts_F4}.
Note that both theorems are true for the groups $F_4(2)$ and $2.F_4(2)$ (in the notation of \cite{Atlas}).
For $q>2$, there is only one finite group of Lie type $F_4$ over a field of order $q$. It is clear, if an element of $N(G,T)$ is a lift of $w$ then it is also a lift to $N_{\overline{G}}(\overline{T})$.

This paper is organized as follows.
In Section 2 we give basic definitions and introduce notation used in this paper. In Section 3 we prove auxiliary results and explain how we use MAGMA in the proofs. Section 4 is devoted to the proofs of Theorems~\ref{th:F4} and \ref{th2:F4}. Finally, in Section 5 we give additional information about maximal tori and lifts in two tables.

\section{Notation and preliminary results}

If a group $G$ is a product of a normal subgroup $N$ and a subgroup $K$ then $K$ is called a {\it supplement} to $N$ in $G$. If $G=NK$ and $N\cap K=1$, then $K$ is called a {\it complement} to $N$ in $G$.

By $q$ we denote some power of a prime $p$. We write $\overline{\F}_p$ for an algebraic closure of a finite field $\F_p$. The symmetric group on $n$ elements is denoted by $S_n$, the cyclic group of order $n$ by $\mathbb{Z}_n$, and the dihedral group of order $2n$ by $D_{2n}$. 
Following~\cite{CarSG}, we write $x^y=yxy^{-1}$ and $[x,y]=y^xy^{-1}$. 

By $\overline{G}$ we denote a simple connected linear algebraic group over $\overline{\F}_p$ 
with the root system $\Phi$ of Lie type $F_4$. We assume that 
$\Delta=\{r_1,r_2,r_3,r_4\}$ is a fundamental system of $\Phi$.

We follow notation from~\cite{CarSG}, in particular, the definitions of elements $x_r(t)$,
$n_r(\lambda)$ ($r\in\Phi$, $t\in\overline{\mathbb{F}}_p$, $\lambda\in\overline{\mathbb{F}}_p^\ast$). In contrast to~\cite{CarSG}, the Magma convention is that $h_r(\lambda) = n_r(-1)n_r(\lambda)$ and we use this definition. According to~\cite{CarSG}, $\overline{G}$ is generated by $x_r(t)$: $\overline{G}=\langle x_r(t)~|~r\in\Phi, t\in\overline{\mathbb{F}}_p\rangle$. The group 
$\overline{T}=\langle h_r(\lambda)~|~r\in\Delta,\lambda\in\overline{\mathbb{F}}_p^\ast \rangle$
is a maximal torus of $\overline{G}$ and $\overline{N}=\langle\overline{T}, n_r~|~r\in\Delta \rangle$, where $n_r=n_r(1)$, is the normalizer of $\overline{T}$ in $\overline{G}$ \cite[\S 7.1, 7.2]{CarSG}. By $W$ we denote the Weyl group $\overline{N}/\overline{T}$ and by $\pi$ the natural homomorphism from $\overline{N}$ onto $W$. Throughout, we denote by $\sigma$
the classical Frobenius automorphism defined on the generators as follows:
$$ x_r(t)\mapsto x_r(t^q), r\in\Phi, t\in\overline{\mathbb{F}}_p^\ast. $$
In particular, we have $n_r^\sigma=n_r$ and $(h_r(\lambda))^\sigma=h_r(\lambda^q)$.
Define the action of $\sigma$ on $W$ in the natural way. Elements $w_1, w_2\in W$ are called {\it $\sigma$-conjugate} if $w_1=(w^{-1})^{\sigma}w_2w$ for some $w$ in $W$.
The following statements hold for $G=\overline{G}_\sigma$.

\begin{prop}{\em\cite[Propositions 3.3.1, 3.3.3]{Car}}\label{torus}.
A torus $\overline{T}^g$ is $\sigma$-stable if and only if $g^{\sigma}g^{-1}\in\overline{N}$. The map $\overline{T}^g\mapsto\pi(g^{\sigma}g^{-1})$ determines a bijection between the $G$-classes of $\sigma$-stable maximal tori of $\overline{G}$ and the $\sigma$-conjugacy classes of $W$.
\end{prop}

\begin{prop}{\em\cite[Lemma 1.2]{ButGre}}\label{prop2.5}.
Let $n=g^{\sigma}g^{-1}\in\overline{N}$. Then $(\overline{T}^g)_\sigma=(\overline{T}_{\sigma n})^g$, where $n$ acts on $\overline{T}$ by conjugation.
\end{prop}

\begin{prop}{\em\cite[Proposition 3.3.6]{Car}}\label{p:normalizer}.
Let $g^{\sigma}g^{-1}\in\overline{N}$ and $\pi(g^{\sigma}g^{-1})=w$. Then $$(N_{\overline{G}}({\overline{T}}^g))_{\sigma}/({\overline{T}}^g)_{\sigma}\simeq C_{W,\sigma}(w)=\{x\in W~|~(x^{-1})^{\sigma}wx=w\}.$$
\end{prop}

By Proposition~\ref{prop2.5}, we have $({\overline{T}}^g)_{\sigma}=(\overline{T}_{\sigma n})^g$ and $(N_{\overline{G}}({\overline{T}}^g))_{\sigma}=(\overline{N}^g)_{\sigma}=(\overline{N}_{\sigma n})^g$. Hence, 

$$C_{W,\sigma}(w)\simeq (N_{\overline{G}}({\overline{T}}^g))_{\sigma}/({\overline{T}}^g)_{\sigma}=(\overline{N}_{\sigma n})^g/(\overline{T}_{\sigma n})^g\simeq\overline{N}_{\sigma n}/\overline{T}_{\sigma n}.$$

\begin{remark}\label{r:nonsplit}
Let $n$ and $w$ be as in Propositions~\ref{prop2.5} and~\ref{p:normalizer}, respectively. Suppose that $n_1=g_1^{\sigma}g_1^{-1}\in\overline{N}$ and $\pi(n_1)=w$. Since $n$ and $n_1$ act on $\overline{T}$ in the same way, we have
$({\overline{T}}^{g_1})_{\sigma}=(\overline{T}_{\sigma n_1})^{g_1}=(\overline{T}_{\sigma n})^{g_1}$ and $(N_{\overline{G}}({\overline{T}}^{g_1}))_{\sigma}=(\overline{N}^{g_1})_{\sigma}=(\overline{N}_{\sigma n_1})^{g_1}$.
Hence, 
$$C_{W,\sigma}(w)\simeq (N_{\overline{G}}({\overline{T}}^{g_1}))_{\sigma}/({\overline{T}}^{g_1})_{\sigma}=(\overline{N}_{\sigma n_1})^{g_1}/(\overline{T}_{\sigma n})^{g_1}\simeq\overline{N}_{\sigma n_1}/\overline{T}_{\sigma n}.$$
Thus, $(\overline{T}^g)_\sigma$ has a complement in its algebraic normalizer if and only if there exists a complement to $\overline{T}_{\sigma n}$ in $\overline{N}_{\sigma n_1}$ for some $n_1$ with $\pi(n_1)=w$. Similarly, if $w$ has a lift to $\overline{N}_{\sigma n_1}$ of order $|w|$
then $w$ has such lift to $N(G,T)$ as well.
\end{remark}

For simplicity of notation, we write $h_r$ for $h_r(-1)$. If $r=r_i$ then $h_i$ stands for $h_{r_i}$ and $n_i$ stands for $n_{r_i}$. Every element $H$ of $\overline{T}$ can be uniquely written in the form $h_{r_1}(\lambda_1)h_{r_2}(\lambda_2)h_{r_3}(\lambda_3)h_{r_4}(\lambda_4)$ and we write briefly $H= (\lambda_1,\lambda_2,\lambda_3,\lambda_4)$.

The group $\mathcal{T}=\langle n_r~|~ r\in\Delta\rangle$ is called {\it the Tits group} \cite{AdamsHe}. Denote $\mathcal{H}=\overline{T}\cap\mathcal{T}$. It is known
that $\mathcal{H}=\langle h_r~|~r\in\Delta\rangle$ and so $\mathcal{H}$ is an elementary abelian group of order $2^4$ such that $\mathcal{T}/\mathcal{H}\simeq W$. 

\begin{remark}\label{r:p=2}
Observe that if $p=2$ then $h_r=1$ for every $r\in\Delta$, in particular, $\mathcal{H}=1$ and $\mathcal{T}\simeq W$. Moreover, the restriction $\tilde{\pi}$ of the homomorphism $\pi$ to $\mathcal{T}$ is an isomorphism between $\mathcal{T}$ and $W$. Let $T$ be a maximal torus corresponding to $w\in W$. Then $n=\tilde{\pi}^{-1}(w)$ is a lift of $w$ to $\overline{N}_{\sigma n}$ of the same order and $\tilde{\pi}^{-1}(C_W(w))$ is a complement to $\overline{T}_{\sigma n}$ in $\overline{N}_{\sigma n}$. Therefore, Theorems~\ref{th:F4}-\ref{th2:F4} hold in the case of even characteristic by Remark~\ref{r:nonsplit}.
\end{remark}

Similarly to~\cite[Theorem 7.2.2]{CarSG}, we have:
\begin{center}
$n_s n_r n_s^{-1}=n_{w_s(r)}(\eta_{s,r}),\quad \eta_{s,r}=\pm1,$
\end{center}
\begin{center}
$n_s h_r(\lambda)n_s^{-1}=h_{w_s(r)}(\lambda).$
\end{center}

We choose values of $\eta_{r,s}$ as follows. 
Let $r\in\Phi$ and $r=\sum\limits_{i=1}^4\alpha_i r_i$. The sum of the coefficients $\sum\limits_{i=1}^4\alpha_i$ is called the {\it height} of $r$. We fix the following total ordering of positive roots: we write $r\prec s$ if either $h(r)<h(s)$ or 
$h(r)=h(s)$ and the first nonzero coordinate of $s-r$ is positive.

Recall that a pair of positive roots $(r,s)$ is called {\it special} if $r+s\in\Phi$ and $r\prec s$. 
A pair $(r,s)$ is called {\it extraspecial} if it is special and for any special pair $(r_1, s_1)$ such that $r+s=r_1+s_1$ one has $r\preccurlyeq r_1$. Let $N_{r,s}$ be the structure constants of the corresponding simple Lie algebra~\cite[Section 4.1]{CarSG}. Then the signs of $N_{r,s}$ may be taken arbitrarily at the extraspecial pairs and then all other structure constants are uniquely determined \cite[Proposition~4.2.2]{CarSG}. In our case, we choose $\sgn(N_{r,s})=+$ for all extraspecial pairs $(r,s)$.
The numbers $\eta_{r,s}$ are uniquely determined by the structure constants \cite[Proposition~6.4.3]{CarSG}.

\section{Preliminaries: calculations}

We use MAGMA~\cite{MAGMA} to calculate products of elements in $\overline{N}=N_{\overline{G}}(\overline{T})$. 
All calculations can be performed using online Magma Calculator~\cite{MC} as well.
At the moment it uses Magma V2.25-5. 
All preparatory commands can be found in \cite{github}.
It is straightforward to verify the defined above ordering gives the same set of extraspecial pairs as in MAGMA. Thus, our calculations in $\overline{N}$ correspond to the ordering and structure constants defined in the previous section. For calculations in $W$, we use GAP \cite{GAP}.
We widely use the following two results for groups $\overline{N}_{\sigma{n}}$ with
$n\in\overline{N}$.

\begin{lemma}{\em \cite[Lemma 1]{GS}}\label{normalizer}
Let $g\in\overline{G}$ and $n=g^\sigma g^{-1}\in\overline{N}$. Suppose that $u\in\mathcal{T}$ and $H\in\overline{T}$. Then

(i) $Hu\in\overline{N}_{\sigma n}$ if and only if   $H=H^{\sigma n}[n,u];$

(ii) If $H\in \mathcal{H}$ then $Hu\in\overline{N}_{\sigma n}$ if and only if   $[n,Hu]=1$.
\end{lemma}

\begin{lemma}\label{commutator}
Suppose that $T$ and $N$ are subgroups of a group $G$ such that $T$ is abelian and 
$N\leq N_G(T)$. Let $a=H_1u_1, b=H_2u_2$, where $H_1, H_2\in T$ and $u_1,u_2\in N$. Then
$$ab=ba \Leftrightarrow H_1^{-1}H_1^{u_2}\cdot u_2u_1u_2^{-1}u_1^{-1}=H_2^{-1}H_2^{u_1}.$$ 
In particular, if $[u_1,u_2]=1$ then
$ab=ba \Leftrightarrow H_1^{-1}H_1^{u_2}=H_2^{-1}H_2^{u_1}.$
\end{lemma}

The following result is our main tool for calculations
of powers and commutators of elements in $\overline{N}$.

\begin{lemma}\label{conjugation} Suppose that $\Phi$ is a root system of type $F_4$
and $r_1$, $r_2$, $r_3$, $r_4$ are fundamental roots. Let $w_r$ be a reflection, where $r\in\Phi$,
and $A=(a_{ij})$ is the matrix of $w_r$ in the basis $r_1$, $r_2$, $r_3$, $r_4$.
Assume that $H=(\lambda_1,\lambda_2,\lambda_3,\lambda_4)$ be an element of $\overline{T}$ and 
\[ B = \left( \begin{matrix} a_{11} & a_{12} & 2a_{13} & 2a_{14} \\ 
a_{21} & a_{22} & 2a_{23} & 2a_{24} \\ 
a_{31}/2 & a_{32}/2 & a_{33} & a_{34} \\ 
a_{41}/2 & a_{42}/2 & a_{43} & a_{44}  \end{matrix}  \right).\]
Then the following claims hold:
 
(i) $H^{n_r}=(\lambda_1',\lambda_2',\lambda_3',\lambda_4')$,
where $\lambda_i'=\lambda_1^{b_{i1}}\cdot\lambda_2^{b_{i2}}\cdot\lambda_3^{b_{i3}}\cdot\lambda_4^{b_{i4}}$;

(ii) if $a,b\in\mathcal{T}$ and 
$H^a=(\lambda_1^{c_{11}}\lambda_2^{c_{12}}\lambda_3^{c_{13}}\lambda_4^{c_{14}},
\lambda_1^{c_{21}}\lambda_2^{c_{22}}\lambda_3^{c_{23}}\lambda_4^{c_{24}},
\lambda_1^{c_{31}}\lambda_2^{c_{32}}\lambda_3^{c_{33}}\lambda_4^{c_{34}},
\lambda_1^{c_{41}}\lambda_2^{c_{42}}\lambda_3^{c_{43}}\lambda_4^{c_{44}}),$
\\$H^b=(\lambda_1^{d_{11}}\lambda_2^{d_{12}}\lambda_3^{d_{13}}\lambda_4^{d_{14}},
\lambda_1^{d_{21}}\lambda_2^{d_{22}}\lambda_3^{d_{23}}\lambda_4^{d_{24}},
\lambda_1^{d_{31}}\lambda_2^{d_{32}}\lambda_3^{d_{33}}\lambda_4^{d_{34}},
\lambda_1^{d_{41}}\lambda_2^{d_{42}}\lambda_3^{d_{43}}\lambda_4^{d_{44}})$
then $$H^{ab}=(\lambda_1^{e_{11}}\lambda_2^{e_{12}}\lambda_3^{e_{13}}\lambda_4^{e_{14}},
\lambda_1^{e_{21}}\lambda_2^{e_{22}}\lambda_3^{e_{23}}\lambda_4^{e_{24}},
\lambda_1^{e_{31}}\lambda_2^{e_{32}}\lambda_3^{e_{33}}\lambda_4^{e_{34}},
\lambda_1^{e_{41}}\lambda_2^{e_{42}}\lambda_3^{e_{43}}\lambda_4^{e_{44}}),$$
where the matrix $(e_{ij})$ is the product of matrices $(c_{ij})$ and $(d_{ij})$;

(iii) $(Hn)^m=(\lambda_1',\lambda_2',\lambda_3',\lambda_4')n^m$,
where $m$ is a positive integer,  $\lambda_i'=\lambda_1^{c_{i1}}\cdot\lambda_2^{c_{i2}}\cdot\lambda_3^{c_{i3}}\cdot\lambda_4^{c_{i4}}$ and $c_{ij}$ are elements of the matrix $\sum\limits_{t=0}^{m-1}B^t$.

\end{lemma}
\begin{proof} First, we prove $(i)$. Observe that $h_s(\lambda)^{n_r}=h_{w_r(s)}(\lambda)$.
Then 
$$H^{n_r}=h_{r_1}(\lambda_1)^{n_r}h_{r_2}(\lambda_2)^{n_r}h_{r_3}(\lambda_3)^{n_r}h_{r_4}(\lambda_4)^{n_r}=h_{w_r(r_1)}(\lambda_1)^{n_r}h_{w_r(r_2)}(\lambda_2)^{n_r}h_{w_r(r_3)}(\lambda_3)^{n_r}h_{w_r(r_4)}(\lambda_4)^{n_r}.$$
Now $w_r(r_i)=a_{1i}r_1+a_{2i}r_2+a_{3i}r_3+a_{4i}r_4$. If  $s=w_r(r_1)$ then
$$\chi_{s,\lambda_1}(a)=\lambda_1^\frac{2(s,a)}{(s,s)}=\lambda_1^\frac{a_{11}2(r_1,a)}{(s,s)}\lambda_1^\frac{a_{21}2(r_2,a)}{(s,s)}\lambda_1^\frac{a_{31}2(r_3,a)}{(s,s)}\lambda_1^\frac{a_{41}2(r_4,a)}{(s,s)}.$$ 
Since $(s,s)=(r_1,r_1)=(r_2,r_2)=2(r_3,r_3)=2(r_4,r_4)$, we have
$$\chi_{s,\lambda_1}(a)=\chi_{r_1,\lambda_1^{a_{11}}}(a)\chi_{r_2,\lambda_1^{a_{21}}}(a)\chi_{r_3,\lambda_1^{a_{31}/2}}(a)\chi_{r_4,\lambda_1^{a_{41}/2}}(a)$$ 
and hence
$h_{w_r(r_1)}(\lambda_1)^{n_r}=(\lambda_1^{a_{11}},\lambda_1^{a_{21}},\lambda_1^{a_{31}/2},\lambda_1^{a_{41}/2})$. Similarly, we obtain
$h_{w_r(r_2)}(\lambda_2)^{n_r}=(\lambda_2^{a_{12}},\lambda_2^{a_{22}},\lambda_2^{a_{32}/2},\lambda_2^{a_{42}/2})$.
From the equalities $$(w_r(r_3),w_r(r_3))=(r_1,r_1)/2=(r_2,r_2)/2=(r_3,r_3)=(r_4,r_4),$$ 
we have $$\chi_{w_r(r_3),\lambda_3}(a)=\chi_{r_1,\lambda_3^{2a_{13}}}(a)\chi_{r_2,\lambda_3^{2a_{23}}}(a)\chi_{r_3,\lambda_3^{a_{33}}}(a)\chi_{r_4,\lambda_3^{a_{43}}}(a).$$
Then $h_{w_r(r_3)}(\lambda_3)^{n_r}=(\lambda_3^{2a_{13}},\lambda_3^{2a_{23}},\lambda_3^{a_{33}},\lambda_3^{a_{43}})$ and, similarly, $h_{w_r(r_4)}(\lambda_4)^{n_r}=(\lambda_4^{2a_{14}},\lambda_4^{2a_{24}},\lambda_4^{a_{34}},\lambda_4^{a_{44}})$. Finally, we find that
\begin{multline*}
H^{n_r}=(\lambda_1^{a_{11}},\lambda_1^{a_{21}},\lambda_1^{a_{31}/2},\lambda_1^{a_{41}/2})(\lambda_2^{a_{12}},\lambda_2^{a_{22}},\lambda_2^{a_{32}/2},\lambda_2^{a_{42}/2})(\lambda_3^{2a_{13}},\lambda_3^{2a_{23}},\lambda_3^{a_{33}},\lambda_3^{a_{43}})(\lambda_4^{2a_{14}},\lambda_4^{2a_{24}},\lambda_4^{a_{34}},\lambda_4^{a_{44}})=\\=
(\lambda_1^{b_{11}}\lambda_2^{b_{12}}\lambda_3^{b_{13}}\lambda_4^{b_{14}},\lambda_1^{b_{21}}\lambda_2^{b_{22}}\lambda_3^{b_{23}}\lambda_4^{b_{24}},\lambda_1^{b_{31}}\lambda_2^{b_{32}}\lambda_3^{b_{33}}\lambda_4^{b_{34}},\lambda_1^{b_{41}}\lambda_2^{b_{42}}\lambda_3^{b_{43}}\lambda_4^{b_{44}}),
\end{multline*}
as claimed.

Now we prove $(ii)$. Since $H^{ab}=(H^{b})^a$, 
we have 
$$H^{ab}=(\lambda_1^{d_{11}}\lambda_2^{d_{12}}\lambda_3^{d_{13}}\lambda_4^{d_{14}},
\lambda_1^{d_{21}}\lambda_2^{d_{22}}\lambda_3^{d_{23}}\lambda_4^{d_{24}},
\lambda_1^{d_{31}}\lambda_2^{d_{32}}\lambda_3^{d_{33}}\lambda_4^{d_{34}},
\lambda_1^{d_{41}}\lambda_2^{d_{42}}\lambda_3^{d_{43}}\lambda_4^{d_{44}})^a.$$
According to (i), if $\lambda_1'$ is the first coordinate of $H^{ab}$ then 
\begin{multline*}
\lambda_1'=(\lambda_1^{d_{11}}\lambda_2^{d_{12}}\lambda_3^{d_{13}}\lambda_4^{d_{14}})^{c_{11}}\cdot(\lambda_1^{d_{21}}\lambda_2^{d_{22}}\lambda_3^{d_{23}}\lambda_4^{d_{24}})^{c_{12}}\cdot(\lambda_1^{d_{31}}\lambda_2^{d_{32}}\lambda_3^{d_{33}}\lambda_4^{d_{34}})^{c_{13}}\cdot(\lambda_1^{d_{41}}\lambda_2^{d_{42}}\lambda_3^{d_{43}}\lambda_4^{d_{44}})^{c_{14}}=\\=
\lambda_1^{d_{11}c_{11}+d_{21}c_{12}+d_{31}c_{13}+d_{41}c_{14}}\lambda_2^{d_{21}c_{12}+d_{22}c_{22}+d_{32}c_{13}+d_{42}c_{14}}\lambda_3^{d_{13}c_{11}+d_{23}c_{12}+d_{33}c_{13}+d_{43}c_{14}}\lambda_4^{d_{14}c_{11}+d_{24}c_{12}+d_{34}c_{13}+d_{44}c_{14}}.
\end{multline*}
So if $(e_{ij})$ are the elements of the product of matrices $(c_{ij})$ and $(d_{ij})$ then
$\lambda_1'=\lambda_1^{e_{11}}\lambda_2^{e_{12}},\lambda_3^{e_{13}}\lambda_4^{e_{14}}$. Similarly, we obtain the expressions for other coordinates of $H^{ab}$.

Now we prove $(iii)$. Observe that $(Hn)^m=H^{n^0}H^{n^1}H^{n^2}\ldots H^{n^{m-1}}n^m$. It follows from $(ii)$ that 

$$H^{n^i}=(\lambda_1^{b^i_{11}}\lambda_2^{b^i_{12}}\lambda_3^{b^i_{13}}\lambda_4^{b^i_{14}},\lambda_1^{b^i_{21}}\lambda_2^{b^i_{22}}\lambda_3^{b^i_{23}}\lambda_4^{b^i_{24}},\lambda_1^{b^i_{31}}\lambda_2^{b^i_{32}}\lambda_3^{b^i_{33}}\lambda_4^{b^i_{34}},\lambda_1^{b^i_{41}}\lambda_2^{b^i_{42}}\lambda_3^{b^i_{43}}\lambda_4^{b^i_{44}}),$$
where $b^i_{jk}$ are element of the matrix $B^i$. Multiplying these expressions for $H^{n^i}$
for $i\in\{0,\ldots,m-1\}$, we obtain the required equation for $(Hn)^m$.
\end{proof}

Since we widely use this lemma in the proofs, we illustrate its applying with the following example.

\noindent{\bf Example.} Consider elements $w=w_3w_2$ and $n=n_3n_2$. 
Then it is easy to see that the matrices $A$ and $B$ for $w$ are the following.
$$A=\left(\begin{matrix}
1 & 0 & 0 & 0 \\
1 & -1 & 1 & 0 \\
2 & -2 & 1 & 1 \\
0 & 0 & 0 & 1
\end{matrix}\right), 
B=\left(\begin{matrix}
1 & 0 & 0 & 0 \\
1 & -1 & 2 & 0 \\
1 & -1 & 1 & 1 \\
0 & 0 & 0 & 1
\end{matrix}\right)  .$$
Let $H=(\lambda_1, \lambda_2, \lambda_3, \lambda_4)$ be an arbitrary element of $\overline{T}$. 
Then by Lemma~\ref{conjugation}(i), we can use the rows of $B$ to compute $H^n$, namely $$H^{n}=(\lambda_1,\lambda_1\lambda_2^{-1}\lambda_3^2,\lambda_1\lambda_2^{-1}\lambda_3\lambda_4,\lambda_4).$$
Now consider $C=B^0+B+B^2+B^3$. We find that
$$C=\left(\begin{matrix}
4 & 0 & 0 & 0 \\
4 & 0 & 0 & 4 \\
2 & 0 & 0 & 4 \\
0 & 0 & 0 & 4
\end{matrix} \right) .$$
It is easy to see that $n^4=h_3$. Now Lemma~\ref{conjugation}(iii) implies that $$(Hn)^4=(\lambda_1^4,\lambda_1^4\lambda_4^4,\lambda_1^2\lambda_4^4,\lambda_4^4)n^4=(\lambda_1^4,\lambda_1^4\lambda_4^4,-\lambda_1^2\lambda_4^4,\lambda_4^4).$$ 
In particular, if $\alpha$ is an element of $\overline{\mathbb{F}}_p$ such that 
$\alpha^2=-1$ then $((\alpha,1,1,1)n)^4=1$.

As an application of Lemma~\ref{conjugation}, we prove an auxiliary result extensively used for proving of non-existence of complements. This is based on the fact that $w=w_{21}w_8w_3w_2$ does not have lifts to $\overline{N}$ of order $|w|$.

\begin{lemma}\label{nolift} Let $w=w_{21}w_8w_3w_2$ and $n=n_{21}n_8n_3n_2$. Consider an element $H\in\overline{T}$.
Then the following statements hold:
\begin{enumerate}[(i)]
\item $(Hn)^4=n^4=h_3$;
\item if $u\in\overline{N}$ then  $(Hn^u)^4=h_3^u$.
\end{enumerate}
\begin{proof} By Lemma~\ref{conjugation}(i),
we find that $H^n=(\lambda_1^{-1}, \lambda_1^{-1}\lambda_2^{-1}\lambda_3^2\lambda_4^{-2}, \lambda_2^{-1}\lambda_3\lambda_4^{-1}, \lambda_4^{-1})$.
Then Lemma~\ref{conjugation}(iii) implies that $(Hn)^4=n^4$.
Using MAGMA, we see that $n^4=h_3$. If $u\in\overline{N}$ then $(Hn^u)^4=((H^{u^{-1}}n)^u)^4=((H^{u^{-1}}n)^4)^u=h_3^u$.
\end{proof}
\end{lemma}

\section{Proofs of Theorems~\ref{th:F4} and~\ref{th2:F4}}
In this section, we prove Theorem~\ref{th:F4}. Since the case $p=2$ was explained in Remark~\ref{r:p=2},
we suppose that $q$ is odd. We assume that $G$ is a finite group of Lie type $F_4$. 
The Dynkin diagram of type $F_4$ is the following.

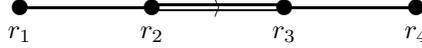
\begin{picture}(330,60)(-140,-30)
\put(0,0){\line(1,0){50}} \put(50,1){\line(1,0){50}}\put(50,-1){\line(1,0){50}}
\put(100,0){\line(1,0){50}}
\put(0,0){\circle*{6}} \put(50,0){\circle*{6}}
\put(100,0){\circle*{6}} \put(150,0){\circle*{6}}
\put(0,-10){\makebox(0,0){$r_1$}}
\put(50,-10){\makebox(0,0){$r_2$}}
\put(75,0){\makebox(0,0){$_\rangle$}}
\put(100,-10){\makebox(0,0){$r_3$}}
\put(150,-10){\makebox(0,0){$r_4$}}
\end{picture}

It is known that in this case the Weyl group $W$ is isomorphic to $GO_4^+\simeq 2^{1+4}:(S_3\times S_3)$ (in the notation of \cite{Atlas}). So $W$ has the central involution $w_0$. There are in total 25 conjugacy classes in $W$.
Among its representatives, there are seven $w$ such that $w$ and $ww_0$ are conjugate in $W$.
All the other 18 classes can be divided on pairs $w^W$ and $(ww_0)^W$ for appropriate elements $w$.
We choose representatives of conjugacy classes according to \cite{Law}. In particular, we have $w_0=w_{21}w_8w_6w_3$. We often use that $w_0=w_1w_3w_{14}w_{2}$ as well. 

There are exactly 25 conjugacy classes of maximal tori in $G$. We number them as in Table~\ref{t:main}. 
To prove Theorem~\ref{th:F4}, we consider each conjugacy class of maximal tori separately.
As an element of $W$ corresponding to a conjugacy class of maximal tori, we use $w$
according to Table~\ref{t:main}. For each $w$ we choose an element $n$ such that $\pi(n)=w$ and
find the minimal order of a supplement to $\overline{T}_{\sigma n}$ in $\overline{N}_{\sigma n}$. If there are no complements to $\overline{T}_{\sigma n}$, we find the minimal order of the lifts of $w$ to $\overline{N}_{\sigma n}$ and present a lift of such order.
All calculations used for the proofs can be found in \cite{github}.
Throughout, we denote by $H=(\lambda_1,\lambda_2,\lambda_3,\lambda_4)$ an arbitrary element of the torus $\overline{T}$.

\textbf{Tori 1 and 17}.
In this case $w=1$ or $w=w_0$, respectively. Suppose that $n=1$ and $M$ is a supplement to $T=\overline{T}_{\sigma n}$ in $\overline{N}_{\sigma n}$. Since $w_{21}w_8w_3w_2\in C_W(w)$,
Lemma~\ref{nolift} implies that $h_3\in M\cap T$.
Denote by $x$ a preimage of $w_4$ in $M$. Then we find that $h_3^{x}=h_3^{n_4}=h_3h_4$.
Since $M\cap T$ is a normal subgroup in $M$, we have $\langle h_3, h_4\rangle\leqslant M\cap T$.

Now we show that there exists a supplement $M$ such that $M\cap T=\langle h_3, h_4\rangle$.
It is well known that
$$C_W(w)\simeq W(F_4)\simeq\langle a,b,c,d~|~ a^2, b^2, c^2, d^2, (ab)^3, [a,c], [a,d], (bc)^4, [b,d], (cd)^3\rangle.$$
Moreover, the reflections $w_1$, $w_2$, $w_3$, and $w_4$ generate $C_W(w)$ and satisfy this set of defining relations.
Denote $$a=h_2n_1, b=h_1n_2, c=h_4n_3, \text{ and } d=h_3n_4.$$
We claim that $M=\langle a, b, c, d \rangle $ is a supplement of 
order $4\cdot|C_W(w)|$ to $\overline{T}_{\sigma n}$. Indeed, using MAGMA, we see 
that $a^2=b^2=c^2=d^2=(ab)^3=[a,d]=[b,d]=(cd)^3=1$
and $[a,c]=(bc)^4=h_3$. Moreover, $M$ normalizes $L=\langle h_3, h_4 \rangle$.
So $L$ is a normal subgroup of $M$ and $M/L$ is a homomorphic image of $C_W(w)$.
Therefore, we have $|M|\leq 4\cdot|C_W(w)|$. On the other hand, we know that $M/(M\cap T)\simeq C_W(w)$
and hence $|M|= |M\cap T|\cdot|C_W(w)|\geq 4\cdot|C_W(w)|$.
Thus, $|M|=4\cdot|C_W(w)|$ and $M\cap T=\langle h_3, h_4 \rangle$, as claimed.

Finally, observe that the identity element of $\overline{N}_{\sigma n}$ is a required lift of $w=1$ and $n_0$ is a lift of order two of $w_0$.

\textbf{Tori 2 and 9.} In this case $w=w_2$ or $w=w_2w_0$, respectively.  Using GAP, we see that $$C_W(w)=\langle w_2 \rangle\times\langle w_0 \rangle \times \langle w_4, w_8, w_{13} \rangle\simeq\Z_2\times\Z_2\times S_4 .$$
If $w=w_2$ then we suppose that $n=n_2$ and $\varepsilon=1$, and if 
$w=w_0w_2$ we suppose that $n=n_0n_2$ and $\varepsilon=-1$.
Denote $T=\overline{T}_{\sigma{n}}$.
By Lemma~\ref{conjugation}, we find that
$$H^{n}=(\lambda_1^\varepsilon, \lambda_1^\varepsilon\lambda_2^{-\varepsilon}\lambda_3^{2\varepsilon}, \lambda_3^\varepsilon, \lambda_4^\varepsilon).$$

Assume that $M$ is a supplement to $T$ such that $|T\cap M|$ is minimal. 
 Calculations in GAP show that $(w_{21}w_8w_3w_2)^{w_7}\in C_W(w)$. Then Lemma~\ref{nolift} implies that
$h_4=h_3^{n_7}$ belongs to $M\cap T$. Denote a preimage of $w_{8}$ in $M$ by $a$. Since $M\cap T$
is a normal subgroup of $M$, we find that $h_4^{a}=h_4^{n_{8}}=h_3h_4\in T\cap M$. 
Therefore, we have $L=\langle h_3, h_4 \rangle\leqslant T\cap M$. 

Suppose that $\varepsilon{q}\equiv 1\negthickspace\pmod4$ and there exists a supplement $M$ to $T$ such that $M\cap T=L$.
Using MAGMA, we see that $[n,n_4]=[n,n_8]=[n,n_{13}]=1$ and hence $n_4$, $n_8$, $n_{13}$ belong to $\overline{N}_{\sigma n}$ by Lemma~\ref{normalizer}. Then there exist elements $H_0=(\mu_i)$, $H_1=(\alpha_i)$, $H_2=(\beta_i)$, and 
$H_3=(\gamma_i)$ of $T$ such that $H_0n_0$, $H_1n_2$, $H_2n_4$, and $H_3n_{13}$ belong to $M$.

Since $M/L\simeq C_W(w)$ and $w_2^2=1$, we have $a^2\in L$.
By Lemma~\ref{conjugation}, we find that
$$(Hn_2)^2=(\lambda_1^2,-\lambda_1\lambda_3^2,\lambda_3^2, \lambda_4^2).$$
Therefore, we have $\alpha_1^2=1$ and $\alpha_1\alpha_3^2=-1$.
Since $H_1\in T$, we infer that
$\alpha_1^{\varepsilon{q}-1}=\alpha_3^{\varepsilon{q}-1}=\alpha_4^{\varepsilon{q}-1}=1$
and $\alpha_1^{\varepsilon{q}}\alpha_2^{-\varepsilon{q}}\alpha_3^{2\varepsilon{q}}=~\alpha_2$.
So $\alpha_2^{\varepsilon{q}+1}=\alpha_1\alpha_3^2=-1$. Similarly, for elements $H_0$, $H_2$, and $H_3$ we find that $\mu_2^{\varepsilon{q}+1}=\mu_1\mu_3^2$ and  $\mu_3^{\varepsilon q-1}=\mu_4^{\varepsilon q-1}=\beta_4^{\varepsilon q-1}=\gamma_4^{\varepsilon q-1}=1$.

Since $w_0$ lies in $Z(C_W(w))$ and $[n_0,n_2]=[n_0,n_4]=[n_0,n_{13}]=1$,
Lemma~\ref{commutator} implies that $H_0^{-1}H_0^{n_2}=H_1^{-2}k_1$,
$H_0^{-1}H_0^{n_4}=H_2^{-2}k_2$, and $H_0^{-1}H_0^{n_{13}}=H_3^{-2}k_3$,
where $k_1,k_2,k_3\in L$. By Lemma~\ref{conjugation}, we have
$$H^{-1}H^{n_2}=(1,\lambda_1\lambda_2^{-2}\lambda_3^2,1,1), H^{-1}H^{n_4}=(1,1,1,\lambda_3\lambda_4^{-2}),
H^{-1}H^{n_{13}}=(1,\lambda_1^2\lambda_3^{-2},\lambda_1^2\lambda_3^{-2},\lambda_1\lambda_3^{-1}).$$

Therefore, we obtain $\mu_1\mu_2^{-2}\mu_3^2=\alpha_2^{-2}$, $\mu_3\mu_4^{-2}=\pm\beta_4^{-2}$, and $\mu_1\mu_3^{-1}=\pm\gamma_4^{-2}$.
Then $(\alpha_2^{-2})^{(\varepsilon{q}+1)/2}=\alpha_2^{-(\varepsilon{q}+1)}=-1$ and hence
$\mu_1^{(\varepsilon{q}+1)/2}\mu_2^{-(\varepsilon{q}+1)}\mu_3^{\varepsilon{q}+1}=-1$.
Since $\mu_2^{\varepsilon{q}+1}=\mu_1\mu_3^2$, we have $\mu_1^{(\varepsilon{q}-1)/2}=-\mu_3^{-(\varepsilon{q}-1)}=-1$.
By assumption, we have $(\varepsilon{q}-1)/2$ is even. Then
 $(\mu_3\mu_4^{-2})^{(\varepsilon{q}-1)/2}=(\pm\beta_4^{-2})^{(\varepsilon{q}-1)/2}=\beta_4^{-(\varepsilon{q}-1)}=1$, so $\mu_3^{(\varepsilon{q}-1)/2}=\mu_4^{\varepsilon{q}-1}=1$.
Finally, we see that $(\mu_1\mu_3^{-1})^{(\varepsilon{q}-1)/2}=(\pm\gamma_4^{-2})^{(\varepsilon{q}-1)/2}=\gamma_4^{\varepsilon{q}-1}=1$; a contradiction with $\mu_1^{(\varepsilon{q}-1)/2}=-1=-\mu_3^{(\varepsilon{q}-1)/2}$.

Suppose now that $\varepsilon{q}\equiv-1\negthickspace\pmod4$. We show that there exists a supplement $M$ such that $T\cap M=\langle h_3, h_4\rangle$. Consider an element $\zeta$ of $\overline{\F}_p$ such that $\zeta^{\varepsilon{q}+1}=-1$. Put
$$H_0=(-1,\zeta^{\frac{\varepsilon{q}+3}{2}},1,1),\qquad H_1=(-1,\zeta,1,1).$$
 
We claim that $M=\langle H_1n_2, H_0n_0, n_4,n_8, n_{13} \rangle$ is a required supplement.
Using MAGMA, we see that $[n,n_2]=[n,n_0]=[n,n_4]=[n,n_8]=[n,n_{13}]=1$, and hence $n_2$, $n_0$, $n_4$, $n_8$, and $n_{13}$ are elements of $\overline{N}_{\sigma n}$. Now we verify that $H_0$ and $H_1$ belong to $T$.
By the above equation for $H^n$, we have
$H_0^n=(-1,-\zeta^{-\varepsilon\frac{\varepsilon{q}+3}{2}},1,1)$ and 
$H_1^n=(-1,-\zeta^{-\varepsilon},1,1)$.
Since $\varepsilon{q}+3\equiv2\negthickspace\pmod4$, we infer that 
$$(-\zeta^{-\varepsilon\frac{\varepsilon{q}+3}{2}})^q=-\zeta^{(1-\varepsilon{q}-1)\frac{\varepsilon{q}+3}{2}}=-\zeta^{\frac{\varepsilon{q}+3}{2}}\zeta^{-(1+\varepsilon{q})\frac{\varepsilon{q}+3}{2}}=-\zeta^{\frac{\varepsilon{q}+3}{2}}(-1)^\frac{\varepsilon{q}+3}{2}=\zeta^{\frac{\varepsilon{q}+3}{2}}.$$

Hence, $$H_0^{\sigma n}=H_0 \text{ and }
H_1^{\sigma n}=(-1,-\zeta^{-\varepsilon{q}},1,1)=(-1,(-\zeta^{-1-\varepsilon{q}})\zeta,1,1)=H_1.$$

Using MAGMA, we see that
$h_3^{n_2}=h_3^{n_0}=h_3^{n_8}=h_3$,
$h_3^{n_4}=h_3^{n_{13}}=h_4^{n_8}=h_3h_4$,
and $h_4^{n_2}=h_4^{n_0}=h_4^{n_4}=h_4^{n_{13}}=h_4$.
It follows that $L$ is a normal subgroup of $M$.

Calculations in MAGMA show that $n_4^2=n_{13}^2=(n_4n_{13})^2=h_4$,
$n_8^2=h_3$, and $(n_4n_8)^3=(n_8n_{13})^3=1$. This implies that $L$ is a subgroup of $\langle n_4,n_8, n_{13}\rangle$ and $\langle n_4,n_8, n_{13}\rangle/L\simeq S_4$.

Now we verify that $\langle H_1n_2, H_0n_0 \rangle\simeq\mathbb{Z}_2\times\mathbb{Z}_2$.
We know that $(Hn_0)^2=1$ for every $H\in\overline{T}$.
Since $H^{n_2}=(\lambda_1,\lambda_1\lambda_2^{-1}\lambda_3^{2}, \lambda_3,\lambda_4)$,
we find that $(H_1n_2)^2=H_1H_1^{n_2}h_2=1$.
Lemma~\ref{commutator} implies that $[H_1n_2,H_0n_0]=1$
if and only if $H_1^{-1}H_1^{n_0}=H_0^{-1}H_0^{n_2}$.
Now $H_1^{-1}H_1^{n_0}=(1,\zeta^{-2},1,1)$.
Using the above equation for $H^{n_2}$, we find that
$$H_0^{-1}H_0^{n_2}=(-1,\zeta^{-\frac{\varepsilon{q}+3}{2}},1,1)(-1,-\zeta^{-\frac{\varepsilon{q}+3}{2}},1,1)=(1,-\zeta^{-(\varepsilon{q}+3)},1,1)=(1,\zeta^{-2},1,1).$$ 
Therefore, $[H_1n_2,H_0n_0]=1$ and hence
$\langle H_1n_2, H_0n_0 \rangle\simeq\mathbb{Z}_2\times\mathbb{Z}_2$.

It remains to prove that the commutators of $H_1n_2$ and $H_0n_0$ 
with $n_4$, $n_8$, $n_{13}$ lie in $L$. We see above that $n_2$ and $n_0$ commute with $n_4$, $n_8$, $n_{13}$.
By Lemma~\ref{commutator}, it suffices to verify that 
$H_1^{-1}H_1^{n_4}$, $H_0^{-1}H_0^{n_4}$, $H_1^{-1}H_1^{n_8}$, $H_0^{-1}H_0^{n_8}$, 
$H_1^{-1}H_1^{n_{13}}$, and $H_0^{-1}H_0^{n_{13}}$ are elements of $L$.
By Lemma~\ref{conjugation}, 
$$H^{-1}H^{n_4}=(1, 1, 1, \lambda_3\lambda_4^{-2}),
H^{-1}H^{n_8}=(\lambda_1^{-2}\lambda_4^2, \lambda_1^{-2}\lambda_4^2, \lambda_1^{-1}\lambda_4, 1),
H^{-1}H^{n_{13}}=(1, \lambda_1^2\lambda_3^{-2}, \lambda_1^2\lambda_3^{-2}, \lambda_1\lambda_3^{-1}).$$
Applying these equations to $H_0$ and $H_1$, we, clearly, obtain elements of $L$.
Therefore, we infer that $M/L$ is a homomorphic image of $\mathbb{Z}_2\times\mathbb{Z}_2\times S_4\simeq C_W(w)$. On the other hand, we know that $M/(T\cap M)\simeq C_W(w)$ and hence $L=M\cap T$. Thus, $M$
is a supplement to $T$ of the minimal order in this case.

No we prove that for every $q$ the group $M=\langle n_0,n_2,n_4,n_8,n_{13}\rangle$ is a supplement to $T$ in $\overline{N}_{\sigma n}$ such that $T\cap M=\langle h_2,h_3,h_4\rangle$. As was shown above, $M$ is a subgroup of $\overline{N}_{\sigma n}$ and $\langle n_4,n_8, n_{13}\rangle/\langle h_3,h_4\rangle\simeq S_4$. Moreover, $n_2$ and $n_0$ lie in the center of $M$. So $\langle h_2\rangle$
is a normal subgroup of $M$. Denote $L=\langle h_2, h_3, h_4 \rangle$. Then $L$ is a normal subgroup of $M$ and $M/L\simeq \mathbb{Z}_2\times\mathbb{Z}_2\times S_4\simeq C_W(w)$. Therefore, we infer that $M$ is a supplement to $T$ in $N$. We see above that any supplement to $T$ includes the subgroup $\langle h_3, h_4 \rangle$
and if $\varepsilon{q}\equiv1\negthickspace\pmod4$ then each supplement has order greater than or equal to $8\cdot|C_W(w)|$.
Thus, if $\varepsilon{q}\equiv1\negthickspace\pmod4$ then $M$ is a supplement to $T$ of the minimal order.

We know that $w_2$ and $w_2w_0$ are elements of order two. It is easy to see that $h_1n_2$ and $h_1n_2n_0$ are required lifts of order two.

\textbf{Tori 3 and 10.} In this case $w=w_3$ or $w=w_0w_3$.
Calculations in GAP show that $C_W(w)$ is isomorphic to $\Z_2\times\Z_2\times S_4$ and has the following presentation
$$C_W(w)\simeq\langle a,b,c,d,e~|~a^2, b^2, c^2, d^2, e^2, [a,b], [a,c], [a,d], [a,e], [b,c], [b,d], [b,e], (cd)^3, (de)^3, (ce)^2 \rangle.$$
Moreover, elements $w_3$, $w_{14}w_{21}w_1$, $w_1$, $w_{16}$, and $w_{14}$ generate $C_W(w)$
and satisfy this set of relations.

Consider $n=n_3$. Suppose that $M$ is a supplement to $T=\overline{T}_{\sigma n}$ in $\overline{N}_{\sigma n}$. 
Using GAP, we see that $(w_{21}w_8w_3w_2)^{w_5}\in C_W(w)$.
It follows from Lemma~\ref{nolift} that $h_3=h_3^{n_5}\in M\cap T$. Therefore, we infer that $|M|\geq 2\cdot|C_W(w)|$.

Now we show that there exists a supplement $M$ such that $M\cap T=\langle h_3 \rangle$.
Let $$a=n_3, b=h_4n_{14}n_{21}n_1, c=h_2h_4n_1, d=h_1n_{16}, \text{ and } e=h_2h_4n_{14}.$$ 
We claim that $M=\langle a, b, c, d, e \rangle$ is a required supplement.

Using MAGMA, we see that $L=\langle h_3 \rangle$ is a normal subgroup of $M$.
Moreover, calculations show that
$$[a,b]=[a,c]=[a,d]=[a,e]=[b,c]=[b,d]=[b,e]=c^2=d^2=e^2=(de)^3=1,$$
$$a^2=b^2=(cd)^3=(ce)^2=h_3.$$ 
By Lemma~\ref{normalizer}, elements $a$, $b$, $c$, $d$, and $e$ belong to $\overline{N}_{\sigma n}$. Then $M/L$ is a homomorphic image of $C_W(w)$ and hence $|M|\leq 2\cdot |C_W(w)|$. On the other hand, we know that $|M|\geq 2\cdot|C_W(w)|$.
Thus, $|M|=2\cdot|C_W(w)|$ and $M\cap T=L$, as claimed.

Observe that elements $w_3$ and $w_0w_3$ have order two. Then $h_2n_3$ and $h_2n_0n_3$ are their lifts of order two.

\textbf{Torus 4.} In this case $w=w_6w_3$ and $w$ is conjugate to $ww_0$ in $W$.
Calculations in GAP show that
$$C_W(w)=\langle w_2, w_3\rangle\times\langle w_{21},w_{24} \rangle\simeq D_8\times D_8.$$

Let $n=n_6n_3$ and $M$ be a supplement to $T=\overline{T}_{\sigma n}$ in $\overline{N}_{\sigma n}$. 
Calculations in GAP show that $w_{21}w_8w_3w_2\in C_W(w)$.
Lemma~\ref{nolift} implies that $h_3\in L=M\cap T$ and hence $|L|\geq2$.
Using MAGMA, we see that $[n,n_2]=[n,h_1n_3]=[n,n_{24}]=1$ and hence $n_2$, $h_1n_3$, and $n_{24}$ belong to $\overline{N}_{\sigma n}$.

Now we prove that $|L|>2$. Assume on the contrary that $L=\langle h_3 \rangle$. Suppose that $a=H_1n_2$ and $b=H_2n_{24}$ are preimages of $w_2$ and $w_{24}$,
where $H_1=(\alpha_1,\alpha_2,\alpha_3,\alpha_4)$ and $H_2=(\beta_1,\beta_2,\beta_3,\beta_4)$.

Since $w_2^2=1$, we have $a^2\in L$. By Lemma~\ref{conjugation}, we find that
$(\alpha_1^2,-\alpha_1\alpha_3^2,\alpha_3^2,\alpha_4^2)\in L$. So $\alpha_1^2=\alpha_4^2=1$ and $\alpha_3^2=-\alpha_1$.

By Lemma~\ref{conjugation}, we obtain $H_2^{-1}H_2^{n_2}=H_1^{-1}H_1^{n_{24}}t$, where $t\in L$.
Then $(1,\beta_1\beta_2^{-2}\beta_3^2,1,1)=(\alpha_1^{-2},\alpha_1^{-3},\alpha_1^{-2},\alpha_1^{-1})t$.
Therefore, we have $\alpha_1=1$ and hence $-\alpha_3^2=\beta_1\beta_2^{-2}\beta_3^2=1$.

Since $w_{24}^2=1$, we have $b^2\in L$.  It follows from $n_{24}^2=h_2h_4$ and Lemma~\ref{conjugation}
that 
$b^2=(1, -\beta_1^{-3}\beta_2^2, \beta_1^{-2}\beta_3^2, -\beta_1^{-1}\beta_4^2)$ and hence
$\beta_1^3=-\beta_2^2$, $\beta_1=-\beta_4^2$. We know that $\beta_1\beta_2^{-2}\beta_3^2=1$, so
$\beta_3^2=\beta_1^{-1}\beta_2^2=-\beta_1^2$.
Since $H_1$ and $H_2$ belong to $T$, we have $H_1^{\sigma n}=H_1$ and $H_2^{\sigma n}=H_2$.
By Lemma~\ref{conjugation}, 
$$H^n=(\lambda_1, \lambda_1^2\lambda_2^{-1}\lambda_4^2, \lambda_1\lambda_3^{-1}\lambda_4^2,\lambda_4).$$
Therefore, we have $H_1=H_1^{\sigma n}=(\alpha_1^q, \alpha_1^{2q}\alpha_2^{-q}\alpha_4^{2q}, \alpha_1^q\alpha_3^{-q}\alpha_4^{2q},\alpha_4^q)$. Then $\alpha_1^{q-1}=\alpha_4^{q-1}=1$ and $\alpha_3=\alpha_1\alpha_3^{-q}\alpha_4^{2}$. So $\alpha_3^{q+1}=\alpha_1\alpha_4^2=1$. On the other hand,
$\alpha_3^{q+1}=(\alpha_3^2)^{(q+1)/2}=(-1)^{(q+1)/2}$ and hence $q+1$ is divisible by~4.
Since $H_2=H_2^{\sigma n}=(\beta_1^q, \beta_1^{2q}\beta_2^{-q}\beta_4^{2q}, \beta_1^q\beta_3^{-q}\beta_4^{2q},\beta_4^q)$, we infer that $\beta_1^{q-1}=\beta_4^{q-1}=1$ and $\beta_3=\beta_1^q\beta_3^{-q}\beta_4^{2q}$. Then $\beta_3^{q+1}=\beta_1\beta_4^2=-\beta_1^2=\beta_3^2$. Therefore, $\beta_3^{q-1}=1$. Since $\beta_1^2=-\beta_3^2$, we have $\beta_1^{q-1}=(-1)^{(q-1)/2}\beta_3^{q-1}=(-1)^{(q-1)/2}=-1$; a contradiction. So $|L|>2$ and hence $|L|\geq4$.

Recall that $w$ and $w_0$ are conjugate in $W$. Consider $n=n_6n_3$ if $q+1$ is divisible by 4 and $n=n_0n_6n_3$ if $q-1$ is divisible by 4. Denote $T=\overline{T}_{\sigma n}$.
Suppose that $\alpha$ is an element of $\overline{\F}_p$ such that $\alpha^2=-1$ and $H_1=(1,1,\alpha,1)$.
We claim that $M=\langle H_1n_2, h_1n_3, n_{21}, n_{24} \rangle$ is a supplement to $T$ in $\overline{N}_{\sigma n}$ of order $4\cdot|C_W(w)|$. First, observe that $n_{24}^2=h_2h_4$ and $(h_1n_3)^2=h_3$. Therefore, $L=\langle h_3, h_2h_4 \rangle\leqslant M\cap T$.
Using MAGMA, we see that $h_3^{n_2}=h_3^{n_3}=h_3^{n_{24}}=h_3^{n_{21}}=h_3$,
$(h_2h_4)^{n_2}=(h_2h_4)^{n_3}=(h_2h_4)^{n_{24}}=h_2h_4$ and $(h_2h_4)^{n_{21}}=h_2h_3h_4$. Therefore, $L$ is a normal subgroup of $M$. 

Now we prove that $H_1\in T$. By Lemma~\ref{conjugation}, 
$$H^{n_6n_3}=(\lambda_1, \lambda_1^2\lambda_2^{-1}\lambda_4^2, \lambda_1\lambda_3^{-1}\lambda_4^2,\lambda_4),$$
$$H^{n_0n_6n_3}=(\lambda_1^{-1}, \lambda_1^{-2}\lambda_2\lambda_4^{-2}, \lambda_1^{-1}\lambda_3\lambda_4^{-2},
\lambda_4^{-1}).$$
If $q\equiv-1\negthickspace\pmod4$ then $H_1^n=(1,1,\alpha^{-1},1)$. So 
$$H_1^{\sigma n}=(1,1,\alpha^{-q},1)=(1,1,\alpha\cdot\alpha^{-1-q},1)=(1,1,\alpha\cdot(-1)^{(q+1)/2},1)=H_1.$$
Similarly, if $q\equiv1\negthickspace\pmod 4$ then $H_1^n=(1,1,\alpha,1)$ and hence 
$$H_1^{\sigma n}=(1,1,\alpha^q,1)=(1,1,\alpha\cdot\alpha^{-1+q},1)=(1,1,\alpha\cdot(-1)^{(q-1)/2},1)=H_1.$$
Using MAGMA, we see that $[n,n_2]=[n,h_1n_3]=[n,n_{21}]=[n,n_{24}]=1$ and hence $M$ is a subgroup of $\overline{N}_{\sigma n}$ by Lemma~\ref{normalizer}. 

Now we verify that $M/L\simeq C_W(w)$. Using MAGMA, we see that
$$(h_1n_3)^2=n_{21}^2=[h_1n_3,n_{21}]=(n_{21}n_{24})^4=h_3, n_{24}^2=[h_1n_3,n_{24}]=h_2h_4.$$
So it suffices to prove that $(H_1n_2)^2$, $[H_1n_2,n_{21}]$, $[H_1n_2,n_{24}]$, and $(H_1n_2h_1n_3)^4$ are elements of $L$. Since $n_2^2=h_2$, we have $(Hn_2)^2=(\lambda_1^2, -\lambda_1\lambda_3^2, \lambda_3^2, \lambda_4^2)$
and hence $(H_1n_2)^2=h_3\in L$. Since $[n_2,n_{21}]=[n_2,n_{24}]=1$, Lemma~\ref{commutator} implies that
$[H_1n_2,n_{21}]$, $[H_1n_2,n_{24}]$ belong to $L$ if $H_1=H_1^{n_{21}}=H_1^{n_{24}}$. By Lemma~\ref{conjugation},
$$H^{n_{21}}=(\lambda_1\lambda_4^{-2}, \lambda_2\lambda_4^{-4}, \lambda_3\lambda_4^{-3}, \lambda_4^{-1}),
H^{n_{24}}=(\lambda_1^{-1}, \lambda_1^{-3}\lambda_2, \lambda_1^{-2}\lambda_3, \lambda_1^{-1}\lambda_4).$$
Applying to $H_1$, we see that $H_1=H_1^{n_{21}}=H_1^{n_{24}}$, as required.
Finally, since $(n_2h_1n_3)^4=h_3$, Lemma~\ref{conjugation} implies that 
$(Hn_2h_1n_3)^4=(Hh_1h_2n_2n_3)^4=(\lambda_1^4,\lambda_1^4\lambda_4^4,-\lambda_1^2\lambda_4^4,\lambda_4^4)$
and hence $(H_1n_2h_1n_3)^4=h_3\in L$. Therefore, $M/L$ is a homomorphic image of $C_W(w)$, in particular
$|M|$ divides $4\cdot|C_W(w)|$. On the other hand, $M/(M\cap T)\simeq C_W(w)$ and $L\leqslant M\cap T$. Thus, $|M|=4\cdot|C_W(w)|$, $L=M\cap T$ and $M$ is a supplement of order $4\cdot|C_W(w)|$ to $T$.

Finally, we know that the orders of $w_6w_3$ and $w_6w_3w_0$ equal two. 
Using MAGMA, we see that $h_1n_6n_3$ and $h_1n_6n_3n_0$ are elements of order two and hence they are  required lifts.

\textbf{Torus 5.} In this case $w=w_{16}w_3$ and $w$ is conjugate to $ww_0$ in $W$.
Using GAP, we see that 
$$C_W(w)=\langle w_3, w_6, w_{16}, w_{24}\rangle\simeq\mathbb{Z}_2\times\mathbb{Z}_2\times\mathbb{Z}_2\times\mathbb{Z}_2.$$
Consider $n=n_{16}n_3$ and $T=\overline{T}_{\sigma n}$. Using MAGMA, we see that $[n,n_3]=[n,h_2n_6]=[n,n_{16}]=1$,
so $n_3$, $h_2n_6$, and $n_{16}$ belong to $\overline{N}_{\sigma n}$ by Lemma~\ref{normalizer}.

Suppose that there exists a complement $K$ to $T$. Denote by $H_1$, $H_2$, and $H_3$ elements of $T$ such that $a=H_1n_3$, $b=H_2h_2n_6$, and $c=H_3n_{16}$ belong to $K$. 
Assume that $H_1=(\alpha_i)$, $H_2=(\beta_i)$, $H_3=(\gamma_i)$, where $1\leq i\leq 4$.

By Lemma~\ref{conjugation}, 
$$H^{n_3}=(\lambda_1,\lambda_2,\lambda_2\lambda_3^{-1}\lambda_4,\lambda_4),
H^{n_6}=(\lambda_1,\lambda_1^2\lambda_2^{-1}\lambda_4^2,\lambda_1\lambda_2^{-1}\lambda_3\lambda_4,\lambda_4),
H^{n_{16}}=(\lambda_1,\lambda_1\lambda_2\lambda_4^{-2},\lambda_1\lambda_3\lambda_4^{-2},\lambda_1\lambda_4^{-1}).$$

Since $n_3^2=h_3$, we infer that $a^2=H_1H_1^{n_3}h_3=(\alpha_1^2,\alpha_2^2,-\alpha_2\alpha_4,\alpha_4^2)$.
So we have $\alpha_1^2=\alpha_2^2=\alpha_4^2=1$ and $\alpha_2\alpha_4=-1$.

Since $n_{16}^2=h_2h_3h_4$, we find that $c^2=(\gamma_1^2,-\gamma_1\gamma_2^2\gamma_4^{-2},-\gamma_1\gamma_3^2\gamma_4^{-2},-\gamma_1)$.
This implies that $\gamma_1=-1$ and $\gamma_2^2=\gamma_4^2=\gamma_3^2$.

Using MAGMA, we see that $[n_3,n_{16}]=1$. Lemma~\ref{commutator} implies that
$H_1^{-1}H_1^{n_{16}}=H_3^{-1}H_3^{n_3}$. From the above equations for $H^{n_3}$ and $H^{n_{16}}$, we infer that
$(1,\alpha_1\alpha_4^{-2},\alpha_1\alpha_4^{-2},\alpha_1\alpha_4^{-2})=(1,1,\gamma_2\gamma_3^{-2}\gamma_4,1)$.
This implies that $\alpha_1\alpha_4^{-2}=1$ and hence $\alpha_1=1$. Then $\gamma_3^2=\gamma_2\gamma_4$. We see above that
$\gamma_3^2=\gamma_2^2$, so $\gamma_2=\gamma_4$.

Using MAGMA, we find that $[h_2n_6,n_{16}]=1$. Lemma~\ref{commutator} implies that
$H_2^{-1}H_2^{n_{16}}=H_3^{-1}H_3^{n_6}$. From the above equations for $H^{n_6}$ and $H^{n_{16}}$, we infer that
$(1,\beta_1\beta_4^{-2},\beta_1\beta_4^{-2},\beta_1\beta_4^{-2})=(1,\gamma_1^2\gamma_2^{-2}\gamma_4^2,\gamma_1\gamma_2^{-1}\gamma_4,1)$. Since $\gamma_2=\gamma_4$ and $\gamma_1=-1$, we have $(1,\beta_1\beta_4^{-2},\beta_1\beta_4^{-2},\beta_1\beta_4^{-2})=(1,1,-1,1)$; a contradiction.

Consider $n=h_1n_{16}n_3$ and $T=\overline{T}_{\sigma n}$. Denote 
$$a=n_3, b=n_6, c=h_1n_{16}, \text{ and } d=h_1h_2n_{24}.$$
We claim that $M=\langle a, b, c, d\rangle$ is a supplement to $T$
in $\overline{N}_{\sigma n}$ of order $2\cdot|C_W(w)|$.
Using MAGMA, we see that $[n,a]=[n,b]=[n,c]=[n,d]=1$
and hence $a$, $b$, $c$, and $d$ belong to $\overline{N}_{\sigma n}$.
Moreover, $a^2=h_3$, so $L=\langle h_3\rangle\leqslant M\cap T$.
Now calculations show that
$$h_3^{a}=h_3^{b}=h_3^{c}=h_3^{d}=h_3,\quad
a^2=b^2=[a,b]=[a,d]=[b,c]=[c,d]=h_3,$$ and $c^2=d^2=[a,c]=[b,d]=1$.
So $L$ is a normal subgroup of $M$ and $M/L$ is a homomorphic image of $C_W(w)$.
On the other hand, we know that $M/(M\cap T)\simeq C_W(w)$.
Thus, $|M|=2\cdot |C_W(w)|$ and $M\cap T=L$, as claimed.

Finally, we see that $(h_1h_2n_{16}n_3)^2=1$ and hence it is a required lift of order two of $w$.

\textbf{Tori 6 and 21.} In this case $w=w_2w_1$ or $w=w_0w_2w_1$, respectively. 
Calculations in GAP show that $$C_W(w)=\langle w_0w_2w_1\rangle\times\langle w_4,w_{19}\rangle\simeq \mathbb{Z}_6\times S_3.$$
Consider $n=n_2n_1$ and denote $a=n_0n_2n_1$, $b=h_3h_4n_4$, and $c=h_4n_{19}$. Using MAGMA, we see that $[n,a]=[n,b]=[n,c]=1$
and hence $a$, $b$, and $c$ belong to $\overline{N}_{\sigma n}$ by Lemma~\ref{normalizer}.
We claim that $K=\langle a,b,c\rangle$ is a complement to $\overline{T}_{\sigma n}$.
Calculations in MAGMA show that $a^{6}=1=[a,b]=[a,c]=1$
and $b^2=c^2=(bc)^3=1$. So $K$ is a homomorphic image of $\mathbb{Z}_6\times S_3$. On the other hand, $C_W(w)\simeq\overline{N}_{\sigma n}/\overline{T}_{\sigma n}\simeq K/(K\cap\overline{T}_{\sigma n})$.
So $K\simeq C_W(w)$ and hence $K$ is a complement to $\overline{T}_{\sigma n}$ in $\overline{N}_{\sigma n}$ as well as to 
$\overline{T}_{\sigma n_0n}$ in $\overline{N}_{\sigma n_0n}$.

\textbf{Tori 7 and 20.} In this case $w=w_3w_4$ or
$w=w_0w_3w_4$, respectively. Calculations in GAP show that
$$C_W(w)=\langle w_0w_3w_4\rangle\times \langle w_1,w_{24}\rangle\simeq \mathbb{Z}_6\times S_3.$$
Consider $n=n_3n_4$. Denote elements $a=n_0n_3n_4$, $b=h_2h_4n_1$, and $c=h_1n_{24}$. 
Using MAGMA, we see that $[n,a]=[n,b]=[n,c]=1$ and hence $a$, $b$, and $c$ belong to $\overline{N}_{\sigma n}$ by Lemma~\ref{normalizer}. We claim that $K=\langle a,b,c\rangle$ is a complement to $\overline{T}_{\sigma n}$. Calculations in MAGMA show that $a^{6}=[a,b]=[a,c]=1$
and $b^2=c^2=(bc)^3=1$. So $K$ is a homomorphic image of $\mathbb{Z}_6\times S_3$. 
On the other hand, it true that $K/(K\cap\overline{T}_{\sigma n})\simeq C_W(w)$.
Therefore, we have $K\simeq C_W(w)$ and hence $K$ is a complement to $\overline{T}_{\sigma n}$ in $\overline{N}_{\sigma n}$ as well as to 
$\overline{T}_{\sigma n_0n}$ in $\overline{N}_{\sigma n_0n}$.

\textbf{Tori 8 and 19.} In this case $w=w_3w_2$ or $w=w_0w_3w_2$, respectively.
Moreover, 
$$C_W(w)=\langle w_3w_2 \rangle\times\langle w_8, w_{24}\rangle\simeq \mathbb{Z}_4\times D_8.$$
Let $n=n_3n_2$ and $\varepsilon=1$ if $w=w_3w_2$ and $n=n_0n_3n_2$
and $\varepsilon=-1$ if $w=w_0w_3w_2$.

Let $M$ be a supplement to $T=\overline{T}_{\sigma n}$ in $\overline{N}_{\sigma n}$. Since $x=w_{21}w_8w_3w_2\in C_W(w)$,
Lemma~\ref{nolift} implies that $h_3\in M$ and hence $h_3\in M\cap T$.

We show that there exists a supplement $M$ such that $M\cap T=\langle h_3 \rangle$.
Let $a=n_3n_2$ and $b=h_2n_8$. Using MAGMA, we see that
$[n,a]=[n,b]=[n, n_{24}]=1$ and hence $a$, $b$, and $n_{24}$ belong to $\overline{N}_{\sigma{n}}$.
Moreover, it is true that $h_3=h_3^{a}=h_3^{b}=h_3^{n_{24}}$ and hence $\langle h_3\rangle$
is a normal subgroup in $\overline{N}_{\sigma{n}}$. Now we see that $a^4=b^2=h_3$ and $[a,b]=1$. So it is sufficient to find an element $c=H_3n_{24}$, where $H_3\in{T}$, such that $[a,c]$, $c^2$, and $(cb)^4$ are elements of $L=\langle h_3 \rangle$. Indeed, let $M=\langle a,b,c\rangle$. Then $M/L$ is a homomorphic image of the group $\mathbb{Z}_4\times D_8$. 
On the other hand, we know that $M/(M\cap T)=C_W(w)$ and $L\leqslant M\cap T$, so $M\cap T=L$ and $M/L\simeq C_W(w)$. 

By Lemma~\ref{conjugation},
$$H^n=(\lambda_1^\varepsilon, \lambda_1^\varepsilon\lambda_2^{-\varepsilon}\lambda_3^{2\varepsilon}, \lambda_1^\varepsilon\lambda_2^{-\varepsilon}\lambda_3^\varepsilon\lambda_4^\varepsilon,
\lambda_4^\varepsilon). $$

Suppose that $\varepsilon{q}\equiv1\negthickspace\pmod 4$. Then put $H_3=(-1,-1,\alpha,1)$, where  $\alpha\in\overline{\F}_p$ such that  $\alpha^2=-1$. Then $H_3^n=(-1,-1,\alpha^\varepsilon,1)$.
Since $\alpha^{\varepsilon{q}}=(\alpha^2)^{\frac{\varepsilon{q}-1}{2}}\alpha=(-1)^{\frac{\varepsilon{q}-1}{2}}\alpha=\alpha$, we have 
 $H_3^{\sigma n}=(-1,-1,\alpha^{\varepsilon{q}},1)=H_3$. Therefore, we infer that $H_3\in T$ and hence $c=H_3n_{24}\in\overline{N}_{\sigma{n}}$.

By the above equation, we find that $H_3^{n_3n_2}=H_3$. It follows from  Lemma~\ref{commutator} that $[a,c]=1$. 

Since $n_{24}^2=h_2h_4$, Lemma~\ref{conjugation} implies that
$$(Hn_{24})^2=(1,-\lambda_1^{-3}\lambda_2^2, \lambda_1^{-2}\lambda_3^2,-\lambda_1^{-1}\lambda_4^2).$$
Applying this equation to $H_3$, we get $c^2=(1,1,-1,1)=h_3\in L$.
Using MAGMA, we see that $(n_{24}h_2n_8)^4=h_3$. Now Lemma~\ref{conjugation} implies that
$$(Hn_{24}h_2n_8)^4=(1, \lambda_1^{-4}\lambda_2^4\lambda_4^{-4}, -\lambda_1^{-2}\lambda_3^4\lambda_4^{-4},1).$$
Therefore, $(cb)^4=(1,1,-1,1)=h_3\in L$. Thus, in this case $a$, $b$, and $c$ generate a supplement of order $2\cdot|C_W(w)|$.

Suppose that $\varepsilon{q}\equiv-1\negthickspace\pmod 4$. Let $\zeta$ be an element of $\overline{\F}_p$ such that
$\zeta^{\frac{q^2-1}{2}}=-1$ and put 
$$H_3=(\zeta^{\varepsilon{q}+1},\zeta^{\frac{3(\varepsilon{q}+1)^2}{4}},
\zeta^{\frac{q^2-1}{4}}\zeta^{\varepsilon{q}+1},\zeta^{\frac{(\varepsilon{q}+1)^2}{4}}).$$
Then we have 
\begin{multline*}
H_3^{n}=(\zeta^{{q}+\varepsilon},
\zeta^{{q}+\varepsilon}\zeta^{\frac{-3\varepsilon(\varepsilon{q}+1)^2}{4}}(-1)\zeta^{2{q}+2\varepsilon},
\zeta^{{q}+\varepsilon}\zeta^{\frac{-3\varepsilon(\varepsilon{q}+1)^2}{4}}\zeta^{\varepsilon\frac{q^2-1}{4}}\zeta^{{q}+\varepsilon}\zeta^{\varepsilon\frac{(\varepsilon{q}+1)^2}{4}},\zeta^{\varepsilon\frac{(\varepsilon{q}+1)^2}{4}})=\\
(\zeta^{{q}+\varepsilon},
-\zeta^{\frac{-3\varepsilon{q}^2+6q+9\varepsilon}{4}},
\zeta^{\frac{-\varepsilon{q}^2+4q+5\varepsilon}{4}},
\zeta^{\frac{\varepsilon{q}^2+2q+\varepsilon}{4}}).
\end{multline*}
Now we find coordinates of $H_3^{\sigma n}$. We see that $$\left(\zeta^{q+\varepsilon}\right)^q=\zeta^{q^2+\varepsilon{q}}=\zeta^{1+\varepsilon{q}},$$
$$(-\zeta^{\frac{-3\varepsilon{q}^2+6q+9\varepsilon}{4}})^q=\zeta^{\frac{-3\varepsilon{q}^3+6q^2+9\varepsilon{q}}{4}}=-\zeta^{\frac{3{q}^2+6\varepsilon{q}+3}{4}}
\zeta^{\frac{3(q^2-1)(1-\varepsilon{q})}{4}}=
-\zeta^{\frac{3(\varepsilon{q}+1)^2}{4}}(-1)^{\frac{3(1-\varepsilon{q})}{2}}=\zeta^{\frac{3(\varepsilon{q}+1)^2}{4}},$$
$$(\zeta^{\frac{-\varepsilon{q}^2+4q+5\varepsilon}{4}})^q=
\zeta^{\frac{-\varepsilon{q}^3+4q^2+5\varepsilon{q}}{4}}=
\zeta^{\frac{q^2-1}{4}}\zeta^{\varepsilon{q}+1}\zeta^{\frac{(q^2-1)(3-\varepsilon{q})}{4}}=
\zeta^{\frac{q^2-1}{4}}\zeta^{\varepsilon{q}+1}(-1)^{\frac{(3-\varepsilon{q})}{2}}=
\zeta^{\frac{q^2-1}{4}}\zeta^{\varepsilon{q}+1},$$
$$(\zeta^{\frac{\varepsilon{q}^2+2q+\varepsilon}{4}})^q=\zeta^{\frac{\varepsilon{q}^3+2q^2+\varepsilon{q}}{4}}=\zeta^{\frac{(\varepsilon{q}+1)^2}{4}}\zeta^{\frac{(q^2-1)(\varepsilon{q}+1)}{4}}=\zeta^{\frac{(\varepsilon{q}+1)^2}{4}}(-1)^{\frac{(\varepsilon{q}+1)}{2}}=\zeta^{\frac{(\varepsilon{q}+1)^2}{4}}.$$ 
Thus, $H_3^{\sigma n}=H_3$ and hence $H_3\in T$.

Since $[a,n_{24}]=1$, Lemma~\ref{commutator} implies that 
$[a,c]\in L$ if $H_3^{-1}H_3^{n_3n_2}\in L$. We have 
$$H^{-1}H^{n_3n_2}=(1,\lambda_1\lambda_2^{-2}\lambda_3^2,\lambda_1\lambda_2^{-1}\lambda_4,1).$$ 
Therefore,
$$H_3^{-1}H_3^{n_3n_2}=(1,\zeta^{\varepsilon{q}+1}\zeta^{\frac{-3(\varepsilon{q}+1)^2}{2}}\zeta^{\frac{q^2-1}{2}}\zeta^{2\varepsilon{q}+2},\zeta^{\varepsilon{q}+1}\zeta^{\frac{-3(\varepsilon{q}+1)^2}{4}}\zeta^{\frac{(\varepsilon{q}+1)^2}{4}},1)=(1,-\zeta^{\frac{-3q^2+3}{2}},\zeta^{\frac{-q^2+1}{2}},1)=h_3.$$
Applying the above equation for $(Hn_{24})^2$ to $H_3$, we get 
$$c^2=(1,-\zeta^{-3\varepsilon{q}-3}\zeta^{\frac{3(\varepsilon{q}+1)^2}{2}},\zeta^{-2\varepsilon{q}-2}\zeta^{\frac{q^2-1}{2}}\zeta^{2\varepsilon{q}+2},-\zeta^{-\varepsilon{q}-1}\zeta^{\frac{(\varepsilon{q}+1)^2}{2}})=(1,-\zeta^{\frac{3q^2-3}{2}},-1,-\zeta^{\frac{q^2-1}{2}})=h_3.$$

Finally, since
$$(Hn_{24}h_2n_8)^4=(1, \lambda_1^{-4}\lambda_2^4\lambda_4^{-4}, -\lambda_1^{-2}\lambda_3^4\lambda_4^{-4},1),$$
we have 
$$(cb)^4=(1,\zeta^{-4\varepsilon{q}-4}\zeta^{3(\varepsilon{q}+1)^2}\zeta^{-(\varepsilon{q}+1)^2},-\zeta^{-2\varepsilon{q}-2}\zeta^{q^2-1}\zeta^{4\varepsilon{q}+4}\zeta^{-(\varepsilon{q}+1)^2},1)=(1,\zeta^{2q^2-2},-\zeta^0,1)=h_3\in L.$$ 
Thus, in this case $a$, $b$, and $c$ generate a supplement of order $2\cdot|C_W(w)|$, as required.

It is easy to see that $w_0w_3w_2$ is conjugate to $w_{21}w_8w_3w_2$ in $W$.
It follows from Lemma~\ref{nolift} that the minimal order of a lift of 
$w_0w_3w_2$ equals eight. 

Now we prove that there exists a lift of order four of $w_3w_2$ if and only if $q\equiv1\negthickspace\pmod4$.
Let $n=n_3n_2$. Using MAGMA, we see that $n^4=h_3$. By Lemma~\ref{conjugation},

$$H^n=(\lambda_1, \lambda_1\lambda_2^{-1}\lambda_3^2, \lambda_1\lambda_2^{-1}\lambda_3\lambda_4, \lambda_4),$$
$$(Hn)^4=(\lambda_1^4, \lambda_1^4\lambda_4^4, -\lambda_1^2\lambda_4^4, \lambda_4^4).$$

Let $q\equiv-1\pmod 4$. Assume that $(H_1n)^4=1$, where $H_1\in T$ and $H_1=(\mu_1,\mu_2,\mu_3,\mu_4)$.
Then $\mu_4^4=1$ and $\mu_1^2\mu_4^4=-1$. So $\mu_1^2=-1$. Since $H_1\in T$, we have $H_1^{\sigma n}=H_1$.
Then $\mu_1^q=\mu_1$ and hence $\mu_1^{q-1}=1$. Since $(q-1)/2$ is odd, we find that $\mu_1^{q-1}=(\mu_1^2)^{(q-1)/2}=(-1)^{(q-1)/2}=-1$; a contradiction. Thus, in this case the minimal order of lifts equal eight.

Let $q\equiv1\pmod 4$ and $\zeta$ be an element of $\overline{\F}_p$ such that $\zeta^{2(q^2+1)}=-1$. 
Consider an element $H_1=(\zeta^{q^2+1},\zeta^{q+1},\zeta,1)$. We claim that $H_1n$ is a lift of order four of $w_3w_2$.
First, we verify that $H_1\in T$. By the above equation, we see that $H_1^n=(\zeta^{q^2+1}, \zeta^{q^2-q+2}, \zeta^{q^2-q+1}, 1)$. Then $H_1^{\sigma n}=(\zeta^{q^3+q}, \zeta^{q^3-q^2+2q}, \zeta^{q^3-q^2+q}, 1)$.
Since 4 divides $q-1$, we have $\zeta^{(q^2+1)(q-1)}=1$ and hence $\zeta^{q^3+q}=\zeta^{q^2+1}$.
Therefore, $H_1^{\sigma n}=(\zeta^{q^2+1}, \zeta^{q^2-q^2+q+1}, \zeta^{q^2-q^2+1}, 1)=H_1$ and $H_1\in T$.

Since $(\zeta^{q^2+1})^2=-1$, the above equation for $(Hn)^4$ implies that $(H_1n)^4=1$, as claimed.

\textbf{Torus 11.} In this case $w=w_2w_1w_{16}$ and $w$ is conjugate to $ww_0$ in $W$.
Moreover, we see that 
$$C_W(w)=\langle w, w_0, w_{10} \rangle\simeq \mathbb{Z}_4\times \mathbb{Z}_2\times \mathbb{Z}_2.$$

Consider $n=n_2n_1n_{16}$ and $T=\overline{T}_{\sigma n}$. Suppose that there exists a complement $K$ to $T$.
Note that $[n,n_0]=1$ and hence $n_0\in\overline{N}_{\sigma n}$.
Consider elements $H_1n$ and $H_2n_0$ of $K$, where $H_1=(\alpha_1,\alpha_2,\alpha_3,\alpha_4)$ and $H_2=(\mu_1,\mu_2,\mu_3,\mu_4)$, that correspond to 
$w$ and $w_0$, respectively.

Since $n^4=h_3h_4$, Lemma~\ref{conjugation} implies that
$$H^n=(\lambda_2\lambda_4^{-2}, \lambda_1\lambda_3^2\lambda_4^{-4}, \lambda_1\lambda_3\lambda_4^{-2}, 
\lambda_1\lambda_4^{-1}),\quad
 (Hn)^4=(\lambda_3^4\lambda_4^{-4}, \lambda_3^8\lambda_4^{-8},-\lambda_3^6\lambda_4^{-6},
-\lambda_3^2\lambda_4^{-2}).$$
Since $w^4=1$, we have $-\alpha_3^2\alpha_4^{-2}=1$ and hence $\alpha_3^2=-\alpha_4^2$.
Lemma \ref{commutator} implies that $H_1^{-1}H_1^{n_0}=H_2^{-1}H_2^n$
and hence $(\alpha_1^{-2},\alpha_2^{-2},\alpha_3^{-2},\alpha_4^{-2})=(\mu^{-1}\mu_2\mu_4^{-2}, \mu_1\mu_2^{-1}\mu_3^2\mu_4^{-4}, \mu_1\mu_4^{-2}, \mu_1\mu_4^{-2})$.
Thus, $\alpha_3^{-2}=\mu_1\mu_4^{-2}=\alpha_4^{-2}$; a contradiction with $\alpha_3^2=-\alpha_4^2$.

So every supplement to $T$ has order at least $2\cdot|C_W(w)|$. Consider $M=\langle n, n_0, n_{10} \rangle$. We claim that $M$ is a supplement to $T$ such that $M\cap T$ is a group of order two.
First, we see that $[n, n_0]=[n,n_{10}]=1$ and so $n_0$ and $n_{10}$ belong to $\overline{N}_{\sigma n}$. Calculations in MAGMA show that 
$(h_3h_4)^n=(h_3h_4)^{n_{10}}=(h_3h_4)^{n_0}=h_3h_4$,
$n_0^2=[n_0,n_{10}]=1$ and $n^4=n_{10}^2=h_3h_4$.
Therefore, $L$ is a normal subgroup of $M$ and $M/L\simeq C_W(w)$. 
So $M\cap T=L$ and $|M|=2\cdot|C_W(w)|$, as claimed.

Let $q\equiv-1\negthickspace\pmod4$. In this case we prove that there are no lifts of $w$ of order four.
Suppose that $(H_1n)^4=1$, where $H_1=(\mu_1,\mu_2,\mu_3,\mu_4)$ belongs to $T$.
Using the above equation for $(Hn)^4$, we find that $\mu_3^2=-\mu_4^2$. Since $H_1\in T$,
we have $H_1^{\sigma n}=H_1$. Therefore, $\mu_1^q\mu_3^q\mu_4^{-2q}=\mu_3$ and  
$\mu_1^q\mu_4^{-q}=\mu_4$. So $\mu_4^{q+1}=\mu_1^q=\mu_3^{1-q}\mu_4^{2q}$ and hence
$\mu_4^{q-1}=\mu_3^{q-1}$. On the other hand, we know that $\mu_3^2=-\mu_4^2$ and hence $\mu_3^{q-1}=(-1)^{(q-1)/2}\mu_4^{q-1}$. Since $(q-1)/2$ is odd, we arrive at a contradiction.

Let $q\equiv1\negthickspace\pmod4$ and $\zeta\in\overline{\F}_p$ such that $\zeta^{q^3+q^2+q+1}=-1$.
Put $H_1=(\zeta^{q^3+1},-\zeta^{1-q}, \zeta^\frac{-q^3-q^2-q+1}{2},\zeta)$.
We claim that $H_1n$ is a lift of $w$ of order four. First, we verify that $H_1\in T$.

By the above equation for $H^n$, we find that 
\begin{multline*}
H_1^n=(-\zeta^{1-q}\cdot\zeta^{-2}, \zeta^{q^3+1}\cdot\zeta^{-q^3-q^2-q+1}\cdot\zeta^{-4}, 
\zeta^{q^3+1}\cdot\zeta^\frac{-q^3-q^2-q+1}{2}\cdot\zeta^{-2},\zeta^{q^3+1}\cdot\zeta^{-1})=\\=(-\zeta^{-1-q},\zeta^{-q^2-q-2},\zeta^\frac{q^3-q^2-q-1}{2},\zeta^{q^3}).
\end{multline*}
Then $H_1^{\sigma n}=(-\zeta^{-q-q^2},\zeta^{-q^3-q^2-2q},\zeta^\frac{q^4-q^3-q^2-q}{2},\zeta^{q^4})$.

Since $\zeta^{q^3+q^2+q+1}=-1$, we see that $-\zeta^{-q-q^2}=\zeta^{q^3+1}$ and $\zeta^{-q^3-q^2-2q}=-\zeta^{1-q}$.
We know that 4 divides $q-1$, therefore, $\zeta^\frac{q^4-1}{2}=\zeta^{(q^3+q^2+q+1)(q-1)/2}=1$.
Then $\zeta^\frac{q^4-q}{2}=\zeta^\frac{1-q}{2}$ and hence $\zeta^\frac{q^4-q^3-q^2-q}{2}=\zeta^\frac{-q^3-q^2-q+1}{2}$.
So $H_1^{\sigma n}=H_1$ and $H_1$ belongs to $T$.

Let $\lambda_3=\zeta^\frac{-q^3-q^2-q+1}{2}$ and $\lambda_4=\zeta$. Then $\lambda_3^2\lambda_4^{-2}=\zeta^{-q^3-q^2-q+1-2}=\zeta^{-q^3-q^2-q-1}=-1$.
Now the above equation for $(Hn)^4$ implies that $(H_1n)^4=1$, as required.

\textbf{Torus 12.} In this case $w=w_{16}w_3w_2$ and
$w$ is conjugate to $ww_0$ in $W$.
Moreover, we see that 
$$C_W(w)=\langle w,w_0,w_{24}\rangle\simeq \mathbb{Z}_4\times \mathbb{Z}_2\times \mathbb{Z}_2.$$

Consider $n=n_{16}n_3n_2$. First, we show that there are no lifts of $w$ of order four. 
Let $a$ be a preimage of $w$ in $\overline{N}_{\sigma n}$.
Then $a=H_1n$, where $H_1=(\lambda_1,\lambda_2,\lambda_3,\lambda_4)$ belongs to $\overline{T}$. 
Using MAGMA, we see that $n^4=h_3$. Then by Lemma~\ref{conjugation} we have
$a^4=(\lambda_1^4,\lambda_1^6,-\lambda_1^4,\lambda_1^2)$.
It is obvious that $a^4$ is not the identity element in $\overline{N}$. 
As a consequence, we infer that there is no complement to $\overline{T}_{\sigma n}$ in $\overline{N}_{\sigma n}$ and the minimal order of the lifts equals eight.
Now we show that there exists a supplement of order $2\cdot|C_W(w)|$. 

Denote $T=\overline{T}_{\sigma n}$.
Let $\zeta$ be an element of $\overline{\F}_p$ such that $\zeta^{\frac{q^4-1}{2}}=-1$ and denote $\eta=\zeta^{-\frac{(q^2+1)(q+1)}{2}}$. Put
$$H_0=(\zeta^{-(q^2+1)(q+1)},\zeta^{-2(q^3+q^2+1)},\zeta^{-(q^3+2q^2+1)},\zeta^{-(q^2+1)}),$$
$$H_1=(-1,\zeta^{q-1},\zeta^{-\frac{(q-1)^2}{2}},-\zeta^{\frac{(q^2+1)(q-1)}{2}}),$$
$$H_2=(-\eta^2,(-1)^{\frac{q-1}{2}}\eta^3,\eta^{\frac{q+3}{2}},\eta).$$

First, we verify that $H_0,H_1,H_2\in T$. 
By Lemma~\ref{conjugation}, 
$$H^n=(\lambda_1,\lambda_1^{2}\lambda_2^{-1}\lambda_3^{2}\lambda_4^{-2},\lambda_1^{2}\lambda_2^{-1}\lambda_3\lambda_4^{-1},\lambda_1\lambda_4^{-1}).$$

Now we apply this equation to $H_0$, $H_1$, and $H_2$.
We find that
\begin{multline*}
H_0^n=(\zeta^{-(q^2+1)(q+1)},\zeta^{-2(q^2+1)(q+1)+2(q^3+q^2+1)-2(q^3+2q^2+1)+2(q^2+1)},\\
\zeta^{-2(q^2+1)(q+1)+2(q^3+q^2+1)-(q^3+2q^2+1)+(q^2+1)},
\zeta^{-(q^2+1)(q+1)+(q^2+1)})=\\=
(\zeta^{-q^3-q^2-q-1},\zeta^{-2q^3-2q^2-2q},\zeta^{-q^3-q^2-2q},\zeta^{-q^3-q}).
\end{multline*}
Since $\zeta^{q^4}=\zeta$, we have 
\begin{multline*}
H_0^{\sigma n}=(\zeta^{-q^4-q^3-q^2-q},\zeta^{-2q^4-2q^3-2q^2},\zeta^{-q^4-q^3-2q^2},\zeta^{-q^4-q^2})=\\=(\zeta^{-q^3-q^2-q-1},\zeta^{-2q^3-2q^2-2},\zeta^{-q^3-2q^2-1},\zeta^{-q^2-1})=H_0.
\end{multline*}
Now we see that 
\begin{multline*}
H_1^{n}=(-1,\zeta^{1-q-(q-1)^2-(q^2+1)(q-1)},-\zeta^{1-q-\frac{(q-1)^2}{2}-\frac{(q^2+1)(q-1)}{2}}, \zeta^{-\frac{(q^2+1)(q-1)}{2}})=\\=(-1,\zeta^{-q^3+1},-\zeta^{\frac{-q^3-q+2}{2}}, \zeta^{-\frac{(q^2+1)(q-1)}{2}}).
\end{multline*}
Since $\zeta^{\frac{q^4-1}{2}}=-1$, we have
$(-\zeta^{\frac{-q^3-q+2}{2}})^q=-\zeta^{\frac{-q^4-q^2+2q}{2}}=-
\zeta^{\frac{-q^4+1}{2}}\zeta^{\frac{-q^2+2q-1}{2}}=\-(-1)^2\zeta^{\frac{-q^2+2q-1}{2}}=
\zeta^{-\frac{(q-1)^2}{2}}$. Moreover,
$(\zeta^{-q^3+1})^q=\zeta^{-q^4+q}=\zeta^{-1+q}$ and 
$(\zeta^{-\frac{(q^2+1)(q-1)}{2}})^q=\zeta^{-\frac{(q^4-1)}{2}+\frac{(q^2+1)(q-1)}{2}}=
-\zeta^{\frac{(q^2+1)(q-1)}{2}}.$ Thus, we have $H_1^{\sigma n}=(-1,\zeta^{-q^3+1},\zeta^{\frac{-q^3-q+2}{2}}, \zeta^{-\frac{(q^2+1)(q-1)}{2}})^q=H_1$.

Finally, we see that 
$$H_2^{n}=(-\eta^2,\eta^4(-1)^{\frac{q-1}{2}}\eta^{-3}\eta^{q+3}\eta^{-2},
\eta^4(-1)^{\frac{q-1}{2}}\eta^{-3}\eta^{\frac{q+3}{2}}\eta^{-1},-\eta)=
(-\eta^2,(-1)^{\frac{q-1}{2}}\eta^{q+2},(-1)^{\frac{q-1}{2}}\eta^{\frac{q+3}{2}},-\eta).$$
Since $\eta^{q-1}=-1$, we have 
$(\eta^{\frac{q+3}{2}})^q=\eta^{\frac{(q+3)(q-1)}{2}}\eta^{\frac{q+3}{2}}=(-1)^{\frac{q+3}{2}}\eta^{\frac{q+3}{2}}$ and
$$H_2^{\sigma n}=(-\eta^2,(-1)^{\frac{q-1}{2}}\eta^{q+2},(-1)^{\frac{q-1}{2}}\eta^{\frac{q+3}{2}},-\eta)^q=(-\eta^2,(-1)^{\frac{q-1}{2}}\eta^{3},(-1)^{\frac{q-1}{2}}(-1)^{\frac{q+3}{2}}\eta^{\frac{q+3}{2}},\eta)=H_2.$$
Therefore, we infer that $H_0,H_1,H_2\in T$.
Using MAGMA, we see that $[n, n_0]=[n, n_{24}]=1$
and hence $a=H_1n_{16}n_3n_2$, $b=H_2n_{24}$, and $c=H_0n_0$ belong to $\overline{N}_{\sigma n}$. 
We claim that $M=\langle a,b,c\rangle$ is a supplement of order $2\cdot|C_W(w)|$ to $T$ such that $M\cap T=\langle h_3 \rangle$.

Since $(Hn)^4=(\lambda_1^4,\lambda_1^6,-\lambda_1^4,\lambda_1^2)$, we have 
$a^4=h_3$. By Lemma~\ref{conjugation},
$$H^{n_{24}}=(\lambda_1^{-1},\lambda_1^{-3}\lambda_2,\lambda_1^{-2}\lambda_3,\lambda_1^{-1}\lambda_4), \quad
(Hn_{24})^2=(1,-\lambda_1^{-3}\lambda_2^{2},\lambda_1^{-2}\lambda_3^{2},-\lambda_1^{-1}\lambda_4^2).$$ 
Hence, we get $b^2=(1,1,\eta^{q-1},1)=h_3$. 

Now we verify that $[b,c]=1$.
By Lemma~\ref{commutator}, it is equivalent to $H_1^{-2}=H_0^{-1}H_0^n$.
Applying the equation for $H^n$, we find that
\begin{multline*}H_0^{-1}H_0^{n}=(\zeta^{(q^2+1)(q+1)},\zeta^{2(q^3+q^2+1)},\zeta^{q^3+2q^2+1},\zeta^{q^2+1})(\zeta^{-q^3-q^2-q-1},\zeta^{-2q^3-2q^2-2q},\zeta^{-q^3-q^2-2q},\zeta^{-q^3-q})=\\(1,\zeta^{2-2q},\zeta^{q^2-2q+1},\zeta^{-q^3+q^2-q+1}),
\end{multline*}
and 
$H_1^{-2}=(1,\zeta^{2-2q},\zeta^{(q-1)^2},\zeta^{-(q^2+1)(q-1)})$. So $H_1^{-2}=H_0^{-1}H_0^n$ and hence
$[a,c]=1$.

Now we verify that $[a,b]\in\langle h_3\rangle$,
Since $\eta^{q-1}=-1$, we have
\begin{multline*}H_2^{-1}H_2^{n}=(-\eta^{-2},(-1)^{\frac{q-1}{2}}\eta^{-3},\eta^{-\frac{q+3}{2}},\eta^{-1})(-\eta^2,(-1)^{\frac{q-1}{2}}\eta^{q+2},(-1)^{\frac{q-1}{2}}\eta^{\frac{q+3}{2}},-\eta)=(1,-1,(-1)^{\frac{q-1}{2}},-1),
\end{multline*}
and 
\begin{multline*}H_1^{-1}H_1^{n_{24}}=(-1,\zeta^{1-q},\zeta^{\frac{(q-1)^2}{2}},-\zeta^{-\frac{(q^2+1)(q-1)}{2}})(-1,-\zeta^{q-1},\zeta^{-\frac{(q-1)^2}{2}},\zeta^{\frac{(q^2+1)(q-1)}{2}})=(1,-1,1,-1).
\end{multline*}
Thus, by Lemma~\ref{conjugation}, we get that $[a,b]$ is equal to either $1$ or $h_3$. Therefore, we have $[a,b]\in\langle h_3 \rangle$. 

It remains to verify that $[c,b]\in\langle h_3 \rangle$.
Using the equations for $H^{n_0}$ and $H^{n_{24}}$, we obtain
\begin{multline*}H_0^{n_{24}}=(\zeta^{(q^2+1)(q+1)},\zeta^{3(q^2+1)(q+1)}\zeta^{-2(q^3+q^2+1)},\zeta^{2(q^2+1)(q+1)}\zeta^{-(q^3+2q^2+1)},\zeta^{(q^2+1)(q+1)}\zeta^{-(q^2+1)})=\\
(\zeta^{(q^2+1)(q+1)},\zeta^{q^3+q^2+3q+1},\zeta^{q^3+2q+1},\zeta^{q^3+q}),
\end{multline*}
and hence
\begin{multline*}H_0^{-1}H_0^{n_{24}}=(\zeta^{(q^2+1)(q+1)},\zeta^{2(q^3+q^2+1)},\zeta^{q^3+2q^2+1},\zeta^{q^2+1})(\zeta^{(q^2+1)(q+1)},\zeta^{q^3+q^2+3q+1},\zeta^{q^3+2q+1},\zeta^{q^3+q})=\\(\zeta^{2(q^2+1)(q+1)},\zeta^{3q^3+3q^2+3q+3},\zeta^{2q^3+2q^2+2q+2},\zeta^{q^3+q^2+q+1}).
\end{multline*}
On the other hand, we have
$$H_2^{-1}H_2^{n_0}=H_2^{-2}=(\eta^{-4},\eta^{-6},\eta^{-(q+3)},\eta^{-2})=
(\eta^{-4},\eta^{-6},-\eta^{-4},\eta^{-2}).$$
Since $\eta=\zeta^{-\frac{(q^2+1)(q+1)}{2}}$, we
infer that $cb=bc$.

Using MAGMA, we see that the group $\langle h_3\rangle$
is a normal subgroup of $M$. So $M/\langle h_3 \rangle$
is a homomorphic image of $C_W(w)$ and hence $|M|\leq 2\cdot|C_W(w)|$. Thus, $|M|=2\cdot|C_W(w)|$, as claimed.

\textbf{Tori 13 and 15.} In this case $w=w_3w_2w_7$ or $w=w_3w_2w_7w_0$, respectively. Calculations in GAP show that 
$$C_W(w)\simeq\langle w_3w_2w_7\rangle\times\langle w_0\rangle\simeq\mathbb{Z}_6\times \mathbb{Z}_2.$$
Consider $n=h_1n_3n_2n_7$. Using MAGMA, we see that $n^{6}=n_0^2=1$ and $[n,n_0]=1$. Therefore, $K=\langle n,n_0\rangle$ is a complement
to $\overline{T}_{\sigma n}$ in $\overline{N}_{\sigma n}$ as well as to 
$\overline{T}_{\sigma n_0n}$ in $\overline{N}_{\sigma n_0n}$.

\textbf{Tori 14 and 16.} In this case $w=w_3w_2w_1$ or $w=w_3w_2w_1w_0$, respectively. Calculations in GAP show that 
$$C_W(w)=\langle w_3w_2w_1\rangle\times\langle w_0 \rangle\simeq\mathbb{Z}_6\times \mathbb{Z}_2.$$
Consider $n=n_3n_2n_1$. Using MAGMA, we see that $n^6=n_0^2=1$ and $[n,n_0]=1$. Therefore, $K=\langle n, n_0\rangle$ is a complement to $\overline{T}_{\sigma n}$ in $\overline{N}_{\sigma n}$ as well as to 
$\overline{T}_{\sigma n_0n}$ in $\overline{N}_{\sigma n_0n}$.

\textbf{Torus 22.} In this case $w=w_8w_{16}w_3w_2$ and $w$ is conjugate to $ww_0$ in $W$.
Calculations in GAP show that $C_W(w)$ is isomorphic to $\SL_2(3):4$ and has the following presentation
$$C_W(w)\simeq \langle a,b,c~|~a^4, b^4, c^3, aba^{-1}b^{-1}, a^3b^2cb^{-1}c^{-2}\rangle. $$ 
Moreover, elements $w_{16}w_8$, $w_2w_3$, and $w_1w_3w_{16}w_{19}$ satisfy these relations and generate $C_W(w)$. 

Let $n=h_2n_8n_{16}n_3n_2$. Denote by $M$ a supplement to $T=\overline{T}_{\sigma n}$ of minimal order.
Calculations in GAP show that $(w_{21}w_8w_3w_2)^{w_{16}}$ belongs to $C_W(w)$. Now Lemma~\ref{nolift} implies that
$h_3=h_3^{n_{16}}$ belongs to $L=T\cap M$. Since $[n, n_1n_3n_{16}n_{19}]=1$, we infer that 
$n_1n_3n_{16}n_{19}\in\overline{N}_{\sigma n}$. Now $L$ is a normal subgroup in $\overline{N}_{\sigma n}$
and hence $h_3h_4=h_3^{n_1n_3n_{16}n_{19}}\in L$. Therefore, we have $|L|\geq4$. 

Now we claim that $a=n_{16}n_8$, $b=h_2n_2n_3$, $c=n_1n_3n_{16}n_{19}$ belong to $\overline{N}_{\sigma n}$ and generate a supplement to $T$ of order $4\cdot|C_W(w)|$.
Using MAGMA, we see that $[n,a]=[n,b]=[n,c]=1$ and hence $a,b,c\in \overline{N}_{\sigma n}$. 
We find that $a^4=h_3$, $b^4=h_3$, $c^3=1$, $[a,b]=h_3$, $a^3b^2cb^{-1}c^{-2}=h_3$.
$h_3^{a}=h_3^{b}=h_3$, $h_3^{c}=h_3h_4$, $h_4^{a}=h_3h_4$, $h_4^{b}=h_3h_4$, $h_4^{c}=h_3$.
Therefore, $\langle h_3, h_4\rangle$ is normal subgroup  of $M=\langle a,b,c\rangle$ and $M/\langle h_3,h_4\rangle$ is a homomorphic image of $C_W(w)$. Thus, $|M|\leq 4\cdot|C_W(w)|$.
On the other hand, we know that $|M|\geq 4\cdot|C_W(w)|$ and hence $M$ is a supplement to $T$ of order $4\cdot|C_W(w)|$.

Finally, we see that $(n_8n_{16}n_3n_2)^4=1$ and hence $n_8n_{16}n_3n_2$ is a required lift of order four.

\textbf{Torus 23.} In this case $w=w_3w_2w_1w_{16}$ and $C_W(w)\simeq\langle w \rangle\simeq \mathbb{Z}_8$.
Moreover, elements $w$ and $ww_0$ are conjugate in $W$.

Consider $n=n_3n_2n_1n_{16}$. Using MAGMA, we see that $n^8=1$. This implies that $K=\langle n\rangle$ is a complement to $\overline{T}_{\sigma n}$ in $\overline{N}_{\sigma n}$.

\textbf{Torus 24.} In this case $w=w_8w_1w_2w_4$ and $C_W(w)\simeq\langle w \rangle\simeq \mathbb{Z}_{12}$.
Moreover, elements $w$ and $ww_0$ are conjugate in $W$.

Consider $n=n_8n_1n_2n_4$. Using MAGMA, we see that $n^{12}=1$ and
hence $K=\langle n\rangle$ is a complement to $\overline{T}_{\sigma n}$ in $\overline{N}_{\sigma n}$.

\textbf{Torus 25 and 18.}
In this case $w=w_6w_1w_9w_4$ or $w=w_0w_6w_1w_9w_4$, respectively.
Furthermore, we have $C_W(w)\simeq\Z_3\times\SL_2(3)$.

Calculations in GAP show that $C_W(w)$ has the following presentation
$$C_W(w)\simeq\langle a,b,c~|~a^3, b^4, c^3, [a,b], [a,c], bcb^{-1}cbc, (c^{-1}b)^3\rangle,$$ 
moreover, elements
$w_1w_4w_2w_{19}$, $w_1w_3w_6w_{20}$, $w_1w_2$ generate $C_W(w)$
and satisfy this set of relations.
Consider $n=n_6n_1n_9n_4$ and denote $a=h_1h_2h_4n_1n_4n_2n_{19}$, $b=h_1h_3h_4n_1n_3n_6n_{20}$, and $c=h_1h_2n_1n_2$.
Using MAGMA, we see that $[n,a]=[n,b]=[n,c]=1$ and hence
$a,b,c\in\overline{N}_{\sigma n}$. Moreover, these elements satisfy
the relations for $C_W(w)$. So $K=\langle a,b,c\rangle$
is a homomorphic image of $C_W(w)$ and hence $K\simeq C_W(w)$.
Thus, $K$ is a complement to $\overline{T}_{\sigma n}$ in $\overline{N}_{\sigma n}$
as well as to $\overline{T}_{\sigma n_0n}$ in $\overline{N}_{\sigma n_0n}$

\section{Tables}

In this section we illustrate results of both theorems in Tables~\ref{t:main} and \ref{lifts_F4}. Let $G=F_4(q)$ and $q$ be odd. 

The first table is devoted to maximal tori of $G$ and Theorem~\ref{th:F4}. There is a bijection between conjugacy classes of maximal tori of $G$ and conjugacy classes of $W$. We choose representatives of the classes of $W$ according to \cite{Law}.
If $T$ is a maximal torus corresponding to a representative $w$ then we write $w$ into the second column of Table~\ref{t:main}. 
The first column contains the number of $T$. If elements $w$ and $ww_0$ are not conjugate then we also indicate in parentheses the number of a maximal tori corresponding to the class $(ww_0)^W$. The third column contains the order of $w$ and the fourth a structure description of $C_W(w)$. The fifth column contains a cyclic structure of $T$ which can be found in a standard way by diagonalizing the matrix  $q\cdot{w}-I$, where $I$ is the identity matrix. Here by $n^k$ we denote the elementary abelian group $\mathbb{Z}_n^k$. Note that the cyclic structure in this column is true for even $q$ as well. Finally, the sixth column contains the minimal possible order of intersections $M\cap T$, where $M$ is a supplement to $T$. Here the entry 4/8 for the second (ninth) torus means that if $q\equiv 1\negthickspace\pmod 4$ ($q\equiv -1\negthickspace\pmod 4$) then the minimal order of the intersection $M\cap T$ is equal to 4, otherwise it is equal to 8.

In the second table we give examples of lifts of $w\in W$ to $N(G,T)$ of order $|w|$ if they exist. The representatives $w$ and their orders are the same as in Table~\ref{t:main}. The third column contains examples of lifts. In the fourth column we write conditions that are necessary and sufficient for the existence of the lifts.
Recall that if there are no lifts of $w$ of order $|w|$ then the minimal order of the lifts equals $2|w|$.

\begin{table}[H]
\centering
\caption{Minimal supplements to maximal tori of $F_4(q)$\label{t:main}}
\begin{tabular}{|l|l|c|l|l|c|}
 \hline
 \No  & $w$ & $|w|$ & Structure of $C_W(w)$ & Cyclic structure of $T$ & Suppl. \\ \hline
  1 (17) & $1$ & 1 & $W(F_4)\simeq GO_4^+$ & $(q-1)^4$ & 4 \\ [2pt]\hline
  2 (9) & $w_2$ & 2 & $\mathbb{Z}_2\times \mathbb{Z}_2\times S_4$ & $(q-1)^2\times(q^2-1)$ &  4/8 \\ [2pt]\hline
  3 (10) & $w_3$ & 2 & $\mathbb{Z}_2\times \mathbb{Z}_2\times S_4$ & $(q-1)^2\times(q^2-1)$ & 2  \\[2pt]\hline
  4  & $w_6w_3$ & 2 & $D_8\times D_8$ & $(q-1)\times(q+1)\times(q^2-1)$ & 4 \\[2pt]\hline
  5  & $w_{16}w_3$ & 2 & $\mathbb{Z}_2\times \mathbb{Z}_2\times \mathbb{Z}_2\times \mathbb{Z}_2$ & $(q^2-1)^2$ & 2 \\[2pt]\hline
  6 (21) & $w_2w_1$ & 3 & $\mathbb{Z}_6\times S_3$ & $(q-1)\times(q^3-1)$ & 1 \\[2pt]\hline
  7 (20) & $w_3w_4$ & 3 & $\mathbb{Z}_6\times S_3$ & $(q-1)\times(q^3-1)$ & 1 \\[2pt]\hline
  8 (19) & $w_3w_2$ & 4 & $\mathbb{Z}_4\times D_8$  & $(q-1)\times(q^2+1)(q-1)$ & 2  \\[2pt]\hline
  9 (2) & $w_{21}w_8w_2$ & 2 & $\mathbb{Z}_2\times \mathbb{Z}_2\times S_4$ & $(q^2-1)\times(q+1)^2$ &  4/8 \\[2pt]\hline
  10 (3) & $w_8w_6w_3$ & 2 & $\mathbb{Z}_2\times \mathbb{Z}_2\times S_4$ & $(q^2-1)\times(q+1)^2$ & 2 \\[2pt]\hline
  11  & $w_2w_1w_{16}$ & 4 &  $\mathbb{Z}_4\times \mathbb{Z}_2\times \mathbb{Z}_2$ & $(q-1,2)\times\cfrac{(q^4-1)}{(q-1,2)}$ &  2  \\[5pt]\hline
  12 & $w_{16}w_3w_2$ & 4 & $\mathbb{Z}_4\times \mathbb{Z}_2\times \mathbb{Z}_2$ & $(q-1,2)\times\cfrac{(q^4-1)}{(q-1,2)}$  & 2 \\[5pt]\hline
  13 (15) & $w_3w_2w_7$ & 6 & $ \mathbb{Z}_6\times \mathbb{Z}_2$ & $(q-1)(q^3+1)$ & 1 \\[2pt]\hline
  14 (16) & $w_3w_2w_1$ & 6 & $ \mathbb{Z}_6\times \mathbb{Z}_2$ & $(q-1)(q^3+1)$ & 1 \\[2pt]\hline
  15 (13) & $w_{16}w_3w_{12}$ & 6 &  $ \mathbb{Z}_6\times \mathbb{Z}_2$ & $(q+1)(q^3-1)$ & 1 \\[2pt]\hline
  16 (14) & $w_{21}w_2w_1$ & 6 &  $ \mathbb{Z}_6\times \mathbb{Z}_2$ & $(q+1)(q^3-1)$ & 1 \\[2pt]\hline
  17 (1) & $w_{21}w_8w_6w_3$ &  2 & $W(F_4)\simeq GO_4^+$ & $(q+1)^4$ & 4 \\[2pt]\hline
  18 (25) & $w_{21}w_2w_1w_{19}$ & 3  & $ \mathbb{Z}_3\times\SL_2(3)$ & $(q^2+q+1)^2$ & 1 \\[2pt]\hline
  19 (8) & $w_{21}w_8w_3w_2$ & 4 & $\mathbb{Z}_4\times D_8$ & $(q^2+1)(q+1)\times(q+1)$ & 2  \\[2pt]\hline
  20 (7) & $w_{21}w_8w_3w_{10}$ & 6 & $\mathbb{Z}_6\times S_3$ & $(q^3+1)\times(q+1)$ & 1 \\[2pt]\hline
  21 (6) & $w_6w_1w_{16}w_3$ & 6 & $\mathbb{Z}_6\times S_3$ & $(q^3+1)\times(q+1)$ & 1 \\[2pt]\hline
  22 & $w_8w_{16}w_3w_2$ & 4 & $\SL_2(3): \mathbb{Z}_4$ & $(q^2+1)^2$ &  4 \\[2pt]\hline
  23 & $w_3w_2w_1w_{16}$ & 8 &  $ \mathbb{Z}_8$ & $(q^4+1)$ & 1 \\[2pt]\hline
  24 & $w_8w_1w_2w_4$ & 12 & $ \mathbb{Z}_{12}$ & $(q^4-q^2+1)$ & 1 \\[2pt]\hline
  25 (18) & $w_6w_1w_9w_4$ & 6 & $ \mathbb{Z}_3\times\SL_2(3)$ & $(q^2-q+1)^2$ & 1 \\
  \hline
\end{tabular}
\end{table}

\begin{table}[H]
\caption{Lifts of elements of $W(F_4)$}\label{lifts_F4}
\centering
\begin{tabular}{|l|c|l|l|}
 \hline
  $w$ & $|w|$ & lift of order $|w|$ & Condition \\ \hline
  $1$ & 1 & 1 &  \\ [1pt]\hline
  $w_2$ & 2 & $h_1n_2$ &  \\ [1pt]\hline 
  $w_3$ & 2 & $h_2n_3$ & \\ [1pt]\hline 
  $w_6w_3$ & 2 & $h_1n_6n_3$ &  \\ [1pt]\hline 
  $w_{16}w_3$ & 2 & $h_1h_2n_{16}n_3$ &  \\ [1pt]\hline 
  $w_2w_1$  & 3  & $n_2n_1$  &  \\ [1pt]\hline
  $w_3w_4$  & 3  & $n_3n_4$  &  \\ [1pt]\hline
  $w_{3}w_2$ & 4 & $(\zeta^{q^2+1},\zeta^{q+1},\zeta,1)n_3n_2$, $\zeta^{2(q^2+1)}=-1$ & $q\equiv1\negthickspace\pmod4$ \\ [1pt] \hline 
  $w_{21}w_8w_2$ & 2 & $h_1n_2n_0$ & \\ [1pt]\hline 
  $w_8w_6w_3$ & 2 & $h_2n_3n_0$ & \\  [1pt]\hline 
  $w_2w_1w_{16}$ & 4 & 
  $(\zeta^{q^3+1},-\zeta^{1-q},\zeta^\frac{-q^3-q^2-q+1}{2},\zeta)n_2n_1n_{16}$,  & $q\equiv1\negthickspace\pmod4$ \\  [1pt] 
     & & $\zeta^{q^3+q^2+q+1}=-1$ & \\ [1pt]\hline
  $w_{16}w_3w_2$ & 4 & -- & \\  [1pt]\hline 
  $w_3w_2w_7$  &  6 & $h_1n_3n_2n_7$  &  \\ [1pt]\hline
  $w_3w_2w_1$  &  6 &  $n_3n_2n_1$ &  \\ [1pt]\hline
  $w_{16}w_3w_{12}$  & 6  & $h_1n_3n_2n_7n_0$  &  \\ [1pt]\hline
  $w_{21}w_2w_1$  & 6  &  $n_3n_2n_1n_0$ &  \\ [1pt]\hline
  $w_{21}w_8w_6w_3$ & 2 & $n_0$ & \\ [1pt]\hline
  $w_{21}w_2w_1w_{19}$  &  3 &  $n_{21}n_2n_1n_{19}$ &  \\ [1pt]\hline
  $w_{21}w_8w_3w_2$ & 4 & -- & \\ [1pt]\hline
  $w_{21}w_8w_3w_{10}$  & 6  & $n_3n_4n_0$  &  \\ [1pt]\hline
  $w_6w_1w_{16}w_3$  & 6  & $n_2n_1n_0$  &  \\ [1pt]\hline
  $w_8w_{16}w_3w_2$ & 4 &  $n_8n_{16}n_3n_2$ & \\[1pt]\hline 
  $w_3w_2w_1w_{16}$ & 8  &  $n_3n_2n_1n_{16}$ &  \\ [1pt]\hline
  $w_8w_1w_2w_4$ &  12 &  $n_8n_1n_2n_4$ &  \\ [1pt]\hline
  $w_6w_1w_9w_4$ &  6 & $n_6n_1n_9n_4$  &  \\ [1pt]\hline
\end{tabular}
\end{table}



\Addresses
\end{document}